\documentclass[%
a4paper,%
arxiv,%
defaults,%
]%
{myclass}

\usepackage[arrow,matrix]{xy}

\usepackage[%
geom,%
bb,%
diff,%
]%
{mymacros}

\newcommand{\lb}{[}
\newcommand{\rb}{]}

\newcommand{\rp}{)}

\newcommand{\pol}{\text{pol}}

\newcommand{\ipv}[2]{\left( #1, #2 \right)}

\newcommand{\ipc}{\ip{\cdot}{\cdot}}

\newcommand{\ad}{\text{ad}}
\newcommand{\per}{\text{per}}
\newcommand{\type}{\text{abc}}
\newcommand{\res}{\text{res}}

\DeclareMathOperator{\Diff}{Diff}

\DeclareMathOperator{\Ad}{Ad}

\DeclareMathOperator{\sign}{sign}

\renewcommand{\T}{\ensuremath{\mathbb{T}}\xspace}

\theoremstyle{\mythmstyle}

\numberwithin{equation}{section}

\mytitle{%
  The co\hyp{}Riemannian Structure of Smooth Loop Spaces
}

\mydate{\today}

\myauthor{%
 Andrew Stacey%
}

\myabstract{%
 We construct a natural co\hyp{}Riemannian structure on the manifold of smooth loops in a Riemannian manifold.
 We show that the smooth loop space of a string manifold is a \emph{per\hyp{}Hilbert\enhyp{}Schmidt} locally equivalent co\hyp{}spin manifold and thus admits a Dirac operator.
}

\mybibliographystyle

\begin{document}

\mymaketitle

\section{Introduction}
\label{sec:intro}

In \cite{math/0809.3104} we introduced the concept of a \emph{co\hyp{}orthogonal structure} on an infinite dimensional vector bundle and showed how this led to the construction of an infinite dimensional Dirac operator.
In brief, whilst an orthogonal structure on a vector bundle defines a Hilbert completion of each fibre, a co\hyp{}orthogonal structure defines a dense Hilbert subspace.
In this paper we shall show that the space of smooth loops in a Riemannian manifold admits a co\hyp{}Riemannian structure\emhyp{}that is, a co\hyp{}orthogonal structure on its tangent bundle\emhyp{}and that this is suitable for the construction of the Dirac operator to work.

We classified co\hyp{}orthogonal and orthogonal structures in \cite{math/0809.3104} according to certain criteria\emhyp{}refining the ``weak\enhyp{}strong'' classification.
In terms of this classification we prove the following theorem.

\begin{thm}
\label{th:coriem}
Let \(M\) be a finite dimensional Riemannian manifold.
Let \(L M \coloneqq \Ci(S^1, M)\) be the space of smooth loops in \(M\).
This space admits a nuclear, locally equivalent, co\hyp{}Riemannian structure with structure group \(\gl_\res\).
\end{thm}

This definition has the following interpretation.
There is a bundle of Hilbert spaces, say \(E \to L M\), with structure group \(\gl_\res\) and a vector bundle map \(E \to T L M\).
There are simultaneous trivialisations of \(E\) and of \(T L M\) with respect to their structure groups such that \(E \to T L M\) locally looks like the inclusion of a fixed Hilbert space in the model space of \(T L M\) (this model space being \(L \R^n\)) and that this inclusion is a nuclear map.
Under this local trivialisation, the inner product on \(E\) is not taken to some fixed inner product on the typical fibre of \(E\).
The group \(\gl_\res\) is the \emph{restricted general linear group} of \cite{apgs}.

Theorem~\ref{th:coriem} refers to \(T L M\) with its standard structure group, namely \(L \gl_n\).
If we are prepared to work with a slightly larger topological group then we can replace the words ``locally equivalent'' by ``locally trivial'' in Theorem~\ref{th:coriem}.
This means that under the simultaneous trivialisation of \(E\) and \(T L M\), the inner product on \(E\) is taken to some fixed inner product on the typical fibre of \(E\).

Having shown this, it is straightforward to show that if \(M\) is a string manifold then \(L M\) has the required structure to define a Dirac operator.

\begin{thm}
Let \(M\) be a finite dimensional string manifold.
Then \(L M\) is an \(S^1\)\enhyp{}equivariant, per\hyp{}Hilbert\enhyp{}Schmidt, locally equivalent, co\hyp{}spin manifold.
\end{thm}

Again, if we are prepared to modify the structure group of the tangent bundle of \(L M\) we have a locally trivial co\hyp{}spin manifold. 

Combined with the work of \cite{math/0809.3104} this yields the following important corollary.

\begin{cor}
Let \(M\) be a finite dimensional string manifold.
Then \(L M\) admits an \(S^1\)\enhyp{}equivariant Dirac operator.
\end{cor}

\medskip

As mentioned above, a co\hyp{}orthogonal structure on an infinite dimensional vector bundle assigns to each fibre a dense subspace with the structure of a Hilbert space\emhyp{}with, almost certainly, a strictly finer topology than the subspace topology.
The classification scheme introduced in \cite{math/0809.3104} measures how much it is possible to locally trivialise this structure, whether these Hilbert spaces fit together to form a bundle in their own right (and with what structure group), and whether the map from the Hilbert space to the larger space has any nice properties such as compactness.
This classification scheme works equally well for orthogonal structures as co\hyp{}orthogonal structures.

To illustrate this classification, the standard Riemannian structure on \(L M\) is a nuclear, locally trivial, orthogonal structure with structure group \(O_\res\) (the restricted \emph{orthogonal} group).
This means that there is a bundle of Hilbert spaces, say \(F \to L M\), with structure group \(O_\res\) and a vector bundle map \(T L M \to F\).
There are trivialisations of \(F\) and of \(T L M\) with respect to their structure groups such that \(T L M  \to F\) locally looks like the inclusion of \(L \R^n\) in a fixed Hilbert space (\(L^2 \R^n\)) and that this inclusion is a nuclear map.
In addition, this local trivialisation maps the inner product to the standard one on \(L^2 \R^n\).

The difference, therefore, is the local triviality of the inner product; this also relates to the fact that we can reduce the structure group of the Hilbert bundle of \(O_\res\) rather than \(\gl_\res\).
The point is that there is only one reasonable loop space on which \(L O_n\) acts isometrically and that is \(L^2 \R^n\).
Of course, we could always reduce the structure group of \(E\) to \(O_\res\) but we couldn't trivialise \(E\) with respect to \(O_\res\) at the same time as trivialising \(T L M\) with respect to its structure group,  \(L \gl_n\).
The purpose of introducing the larger structure group that we mention above is to enable us to trivialise \(E\) with respect to \(O_\res\) and \(T L M\) with respect to this larger group simultaneously.

\medskip

We shall give the construction in three stages.

The first stage, in section~\ref{sec:linear}, is the linear case.
As the co\hyp{}orthogonal structure that we wish to construct is intended to have some local triviality properties, we need to decide on a reference structure.
That is, we need to choose a fixed Hilbert subspace of \(L \R^n\).
This is not difficult: we take those loops which are analytic on an annulus of radius \((r^{-1},r)\) for some fixed \(r \in (1,\infty)\) and are square integrable on the boundary.
We shall denote this space by \(L_r^2 \R^n\).

We also need to investigate what group acts on this space.
It is easy to show that \(L \gl_n(\R)\) does not preserve any Hilbert subspace of \(L \R^n\) but we can reduce \(L \gl_n(\R)\) to \(L_\pol O_n\), the group of polynomial loops in \(O_n\), which does.
We also show that the action of \(L_\pol O_n\) on \(L_r^2 \R^n\) factors through the restricted general linear group.
There is an obvious isomorphism of \(L_r^2 \R^n\) with \(L^2 \R^n\) and we describe how this fits into the mixture.

The second stage, in section~\ref{sec:universal}, is the universal case.
Let \(G\) be a connected, compact Lie group.
Let \(\Omega G\) be the based loop group of \(G\).
This is a Lie group and its classifying space is again \(G\).
Now \(\Omega_\pol G\), the polynomial loop group of \(G\), is homotopy equivalent to \(\Omega G\) and so \(G\) is also the classifying space for \(\Omega_\pol G\).
This means that there is a universal \(\Omega_\pol G\)\enhyp{}bundle over \(G\) and, consequently, an \(L_\pol G\)\enhyp{}bundle and various associated vector bundles.
We shall give explicit constructions of these.
The difficult part of this section is proving local triviality: as our constructions are explicit we cannot simply appeal to the homotopy equivalence \(\Omega_\pol G \simeq \Omega G\) to deduce that they are the required universal bundles.

The reason for doing the universal case is simple.
Let \(E \to M\) be a finite dimensional vector bundle with structure group \(G\).
A connection on \(E\) defines a holonomy map, \(\Omega M \to G\).
This is an explicit classifying map for the bundle \(\Omega E \to \Omega M\).
That is, there is a \(\Omega \R^n\)\enhyp{}bundle over \(G\) which pulls\hyp{}back to \(\Omega E\) under the holonomy map.
Having constructed a sub\enhyp{}\(\Omega_\pol \R^n\)\enhyp{}bundle we can pull this back as well and obtain a subbundle \(\Omega_\pol E\) of \(\Omega E\) which, on fibres, looks like the inclusion of \(\Omega_\pol \R^n\) in \(\Omega \R^n\).

Thus by starting with a vector bundle with connection, we obtain an explicit classifying map and so the universal construction defines a specific construction on the loop bundle.
Compare this with the situation of orthogonal structures on finite dimensional vector bundles.
The inclusion \(O_n \subseteq \gl_n\) is a homotopy equivalence but choosing a universal orthogonal structure on the vector bundle over \(B \gl_n\) does not automatically define an orthogonal structure on every rank \(n\) vector bundle\emhyp{}it merely says that these exist.

The final stage, in section~\ref{sec:free}, is to adapt this to free loop spaces.
This is reasonably straightforward.
It uses the fact that the free loop space is a fibre bundle over the original manifold with fibre the based loop space.
The strategy is to apply the above construction to the fibres of this bundle and twist it suitably over the base.

\medskip

Finally, let us give a formula for the resulting inner product.
Let \(M\) be a finite dimensional Riemannian manifold.
Let \(\gamma \colon S^1 \to M\) be a smooth loop in \(M\).
The tangent space of \(L M\) at \(\gamma\) is naturally isomorphic to \(\Gamma(\gamma^* T M)\), the space of sections of \(\gamma^* T M \to S^1\).
Let \(D_\gamma \colon \Gamma(\gamma^* T M) \to \Gamma(\gamma^* T M)\) be the covariant differential operator defined by the Levi\hyp{}Civita connection on \(M\).
This is a skew\hyp{}adjoint operator with spectrum of the form \((i s_1, \dotsc i s_n) + i \Z^n\) for some real numbers \(s_j\).
Although \(D_\gamma\) is a differential operator when \(\Gamma(\gamma^* T M)\) is considered as a space of sections, when \(\Gamma(\gamma^* T M)\) is considered as the tangent space of \(T L M\), \(D_\gamma\) is more correctly viewed as a linear operator.

Let \(\cos D_\gamma\) be the densely\hyp{}defined operator on \(\Gamma(\gamma^* T M)\) defined using the power series for \(\cos\).
For \(\alpha, \beta \in \Gamma(\gamma^* T M)\) such that \(\cos D_\gamma \alpha\) and \(\cos D_\gamma \beta\) are defined, let
\[
  \ip{\alpha}{\beta} = \int_{S^1} \ip{(\cos D_\gamma \alpha)(t)}{(\cos D_\gamma \beta)(t)}_{T_{\gamma(t)} M} d t.
\]
The Hilbert subspace of \(T_\gamma L M = \Gamma(\gamma^* T M)\) is the subspace on which this formula makes sense.

\subsection{Acknowledgements and History}
\label{sec:introack}

The central idea of this paper\emhyp{}the construction of the polynomial loop bundle\emhyp{}places this paper as the latest in a loosely defined series: \cite{jm2}, \cite{rcas}, and \cite{as2}.
In the first of these, Morava attempted to construct an isomorphism for an almost complex manifold \(M\) between the tangent bundle of the loop space, \(L T M\), and a bundle of the form
 \(e_1^* T M \otimes_\C L \C\).
Here, \(e_1 \colon L M \to M\)
is the map which evaluates a loop at time \(1\) and every bundle is considered to be complex.
The argument broke down at one crucial step and the paper \cite{jm2} had to be withdrawn, \cite{jm9}\footnote{Although withdrawn, \cite{jm2} is still available from the website of the journal}.
The papers \cite{rcas} and \cite{as2} grew out of considering the question as to when that crucial step could be made to work.
This was found to be highly restrictive and implied, for example, that the tangent bundle of the based loop space of \(M\) was trivial.

One consequence which would follow from the existence of an isomorphism
 \(L T M \cong e_1^* T M \otimes_\C L \C\)
would be the existence of a subbundle modelled on the polynomial loop space.
In fact, for any class of loops there would be a bundle with the appropriate fibre constructed as
 \(e_1^* T M \otimes_\C L^\alpha \C\).
The construction in this paper was inspired by that of Morava.
The basic idea of using parallel transport to select a suitable subbundle of \(T L M\) comes directly from \cite{jm2}.
The bulk of this paper is concerned with proving the local triviality of that bundle (this is something that Morava did not have to consider as his bundle was\emhyp{}supposed to be\emhyp{}isomorphic to \(e_1^* T M\) and thus automatically locally trivial).

The author would like to thank Rafe Mazzeo, Ralph Cohen, and Eldar Straume for helpful conversations, to thank Gerd Laures for asking an interesting question, and to acknowledge the encouragement of Jack Morava.

\subsection{Notation}
\label{sec:notation}

In this paper we have two views of the circle.
One is as the domain of loops, the other as a Lie group.
We regard loops as periodic paths from \R and thus wish to identify the domain of loops with \(\R/\Z\).
When thinking of the circle as a Lie group, we think of it as \(U_1\) sitting inside \(M_1(\C) = \C\).
We shall use the notation \(S^1\) for \(\R/\Z\) and \T for \(U_1\).
We shall write \(t\) for the parameter in \(S^1\) and \(z\) in \T, with relationship \(z = e^{2 \pi i t}\).

For a finite dimensional smooth manifold \(M\) we define the loop space, \(L M\), and path space, \(P M\), as
\begin{align*}
 L M \coloneqq \Ci(S^1, M), \\
 P M \coloneqq \Ci(\R, M).
\end{align*}
As we are viewing the circle as a quotient of \(\R\) we have a natural inclusion \(L M \to P M\) as the subspace of periodic paths with period \(1\).
Within \(L M\) we have a copy of \(M\) as the constant loops.
Thus we can regard \(M\) as a subspace of \(L M\) and both as subspaces of \(P M\).

%\newpage
\section{The Linear Case}
\label{sec:linear}

In this section we shall consider the linear case.
We shall find suitable subspaces of \(L \R^n\) and of \(L \C^n\).
To extend this to the bundle case we shall need to examine various group actions on these.

\subsection{Polynomials and Polarisations}
\label{sec:polypol}

Let us start by introducing the main groups that we will be interested in.
The following definitions are standard.
The main reference for this section is \cite{apgs}.

\begin{defn}
Let \(n \ge 1\), let
  \(H \coloneqq L^2(S^1, \C^n)\),
 and define \(J_n \colon H \to H\) to be the operator
  \(J_n(v z^n) = i \sign(n) v z^n\).
 Define the \emph{restricted general linear group}, \(\gl_\res(H)\), as the subgroup of \(\gl(H)\) consisting of those operators \(T\) for which the commutator \(\lb T, J \rb\) is Hilbert\enhyp{}Schmidt.
 Define the \emph{restricted unitary group}, \(U_\res(H)\), as the intersection
  \(U(H) \cap \gl_\res(H)\).

 Let
  \(H_\R \coloneqq L^2(S^1, \R^n)\)
 and identify \(H\) with \(H_\R \otimes \C\) in the obvious way; this allows us to view \(\gl(H_\R)\) and \(O(H_\R)\) as subgroups of \(\gl(H)\).
 Define the \emph{restricted general linear group} of \(H_\R\), \(\gl_\res(H_\R)\), and the \emph{restricted orthogonal group} of \(H_\R\), \(O_\res(H_\R)\), as, respectively,
  \(\gl(H_\R) \cap \gl_\res(H)\)
 and
  \(O(H_\R) \cap \gl_\res(H)\).

 The group \(O_\res(H_\R)\) has two connected components, we denote the component of the identity by \(S O_\res(H_\R)\).
\end{defn}

Recall that an operator \(T \colon H_1 \to H_2\) between Hilbert spaces is \emph{Hilbert\enhyp{}Schmidt} if for some, and hence every, orthonormal basis \(\{e_i\}\) for \(H\) the series
 \(\sum_i \norm[T e_i]^2\)
is absolutely convergent.

The groups introduced above are not topologised as subgroups of \(\gl(H)\), rather the topology is strengthened to take into account the Hilbert\enhyp{}Schmidt norm on the commutators.
The details can be found in \cite{apgs}.
Although we have defined these groups using a specific operator on a specific separable Hilbert space, the groups so defined are\emhyp{}up to isomorphism\emhyp{}independent of these choices.
It is therefore customary to drop the ``\(H\)'' or ``\(H_\R\)'' from the notation when this does not lead to a loss of clarity.

These groups have a strong relationship to loop groups which we shall now introduce.
The smooth loop space of a compact Lie group is again a group and we shall refer to it as the \emph{smooth loop group}.
The definition of the polynomial loop group appears in~\cite[\S 3.5]{apgs} and we repeat it here.

\begin{defn}
 \label{def:polgrp}
 Let \(G\) be a compact Lie group.
 Fix an embedding of \(G\) as a subgroup of \(U_n\) for some \(n\).
 This exhibits \(G\) as a submanifold of \(M_n(\C)\).
 The \emph{polynomial loop group} of \(G\), \(L_\pol G\), is defined as the space of those loops in \(G\) which, when expanded as a Fourier series in \(M_n(\C)\), are finite Laurent polynomials.
 The group of based loops, \(\Omega_\pol G\), is the subgroup of \(L_\pol G\) of loops \(\gamma\) with \(\gamma(0) = 1_G\).
\end{defn}

\begin{remark}
 The following comments appear in~\cite[\S 3.5]{apgs}:

 \begin{enumerate}
 \item
  The choice of the embedding of \(G\) in \(U_n\) is immaterial.

 \item
  The space \(L_\pol G\) is the union of the subspaces \(L_{\pol,N} G\) consisting of those loops with Fourier series of the form:
  \[
   \sum_{k = -N}^N \gamma_k z^k.
  \]
  These spaces are naturally compact.
  The topology on \(L_\pol G\) is the direct limit topology of this union.

 \item
  The free polynomial loop group is the semi\hyp{}direct product of the based polynomial loop group and the constant loops.

 \item
  The group \(L_\pol G\) does not have an associated Lie algebra, although the Lie algebra \(L_\pol \mf{g}\) is often linked to it.

 \item
  If \(G\) is semi\hyp{}simple then \(L_\pol G\) is dense in \(L G\).

 \item
  In the case of the circle, \(\Omega_\pol S^1 = \Z\)
  and so
   \(L_\pol S^1 = S^1 \times \Z\).
 \end{enumerate}
\end{remark}

The following is \cite[proposition~8.6.6]{apgs}:

\begin{proposition}
 If \(G\) is semi\hyp{}simple then the inclusion
  \(\Omega_\pol G \to \Omega G\)
 is a homotopy equivalence. \noproof
\end{proposition}

Since
 \(L_\pol G \cong \Omega_\pol G \times G\)
and
 \(L G \cong \Omega G \times G\)
as spaces (although not generally as groups), this holds for the unbased loops as well.

In passing, let us observe that the topology on \(L_\pol G\) makes questions of continuity very easy to determine.
Suppose that a representation of \(L_\pol G\) factors through a representation of the polynomial algebra \(L_\pol M_n(\C)\) for some \(n\).
It is therefore automatically continuous as \(L_\pol M_n(\C)\) is topologised as the direct limit of its finite dimensional subspaces and so any linear map out of \(L_\pol M_n(\C)\) into a topological vector space is continuous.

The following result is the starting point of \cite{apgs}.

\begin{proposition}
 The natural action of \(L \gl_n(\C)\) on \(H = L^2(S^1, \C^n)\) factors through a homomorphism
  \(L \gl_n(\C) \to \gl_\res(H)\). \noproof
\end{proposition}

We therefore see that if \(G\) acts on a finite dimensional real or complex vector space \(V\) then the smooth and polynomial loop groups act on
 \(H \coloneqq L^2(S^1, V)\)
via a homomorphism into \(\gl_\res(H)\).

\subsection{Linear Co\hyp{}Orthogonal Structures}
\label{sec:linco}

Our main theorem states that if \(M\) is a finite dimensional Riemannian manifold then \(L M\) admits a nuclear, locally equivalent, co\hyp{}Riemannian structure with structure group \(\gl_\res\).
This assigns to each tangent space of \(L M\) a linear map from a Hilbert space which is injective and has dense image.
The term ``locally equivalent'' means that these look locally like the inclusion of some standard Hilbert space in \(L \R^n\).
In this section we shall define this standard space.

It is also important to consider group actions.
The linear interpretation of the phrase ``with structure group \(\gl_\res\)'' is that there is some subgroup of \(L \gl_n\) acting on \(L \R^n\) which also acts on the Hilbert subspace and that this action is via \(\gl_\res\).
That we can only say ``locally equivalent'' and not ``locally trivial'' means that this action is not by isometries.

We shall define a \(1\)\enhyp{}parameter family of Hilbert subspaces indexed by \(r \in (1,\infty)\).
Contrast this with the opposite structure, that of finding a Hilbert completion of \(L \R^n\).
There is an obvious family of such completions, \(L^{k,2} \R^n\),  indexed by \(k \in \N_0\).
In this case the completion corresponding to \(k = 0\) has a mathematical advantage unshared by the others: the action of the loop group \(L O_n\) is by isometries.

We start by defining our space of interest and a particularly important operator.

\begin{defn}
\label{def:cosds}
 Let \(V\) be a complex vector space of dimension \(n\) with an inner product.
 Let \(r \in (1, \infty)\).
 Let \(L_\pol V\) be the space of \(V\)\enhyp{}valued polynomial loops, and \(L^2_r V\) the space of loops in \(V\) which extend holomorphically over an annulus of radii \((r^{-1},r)\) and are square integrable on the boundary.

 Let
  \(s \coloneqq (\{v_1, \dotsc, v_n\}, \{s_1, \dotsc, s_n\})\)
 where \(\{v_1, \dotsc, v_n\}\)
 is a basis for \(V\) and
  \(s_1, \dotsc, s_n \in \R\).
 Let
  \(D_s \colon L_\pol V \to L_\pol V\)
 be the operator defined by
  \(D_s v_j z^p = i(p+s_j) v_j z^p\).
 Define
 \[
  \cos_r D_s \coloneqq \sum_{j \ge 0} \frac{(-1)^j}{(2j)!} (\log (r) D_s)^{2 j}.
 \]
\end{defn}

We state our main result on these.

\begin{theorem}
 \label{th:ldeiso}
 The group \(L_\pol U(V)\) acts continuously on \(L_r^2 V\) and this action factors through \(\gl_\res\).

 The map \(\cos_r D_s\) extends to an isomorphism \(L_r^2 V \to L^2 V\) which intertwines the actions of \(\gl_\res\) on each space.
\end{theorem}

\begin{proof}
 Taking Fourier coefficients allows us to identify all of these spaces with certain spaces of \Z{}\hyp{}indexed sequences in \(V\).
 Smooth loops, \(L V\), corresponds to the space of rapidly decreasing sequences; \(L^2 V\) to the space of square integrable sequences; polynomial loops, \(L_\pol V\), corresponds to the space of finite sequences; and \(L^2_r V\) corresponds to the space of sequences \((a_p)\) for which the sequence \((r^{\abs{p}} a_p)\) is square integrable.

 To show that the polynomial loop group, \(L_\pol U(V)\), acts continuously on all of these spaces it is sufficient to show that the algebra \(L_\pol \End(V)\) acts.
 This is generated as a topological algebra by the subalgebra \(\End(V)\) and the operator which multiplies by \(z\).
 The algebra \(\End(V)\) obviously acts continuously on all of the spaces.
 The operation corresponding to multiplication by \(z\) is the shift operator on the sequence spaces which is also obviously continuous.

 For \(j, p \in \Z\) with \(1 \le j \le n\), let \(e_{p,j}\) be the sequence with \(v_j\) in the \(p\)th slot and zero elsewhere.
 These form a topological basis for all of the above spaces.
 On this basis, \(D_s\) is
 \[
  e_{p,j} \mapsto i(p + s_j) e_{p,j}.
 \]
 Hence \(\cos_r D_s\) is
 \[
  e_{p,j} \mapsto \cos \big(i (p + s_j) \log (r)\big) e_{p,j} = \cosh \big((p + s_j)\log(r)\big) e_{p,j}.
 \]

 Elementary analysis shows that for all \(x,t \in \R\)
 \[
  \cosh(t \log(r)) \ge \frac{\cosh\big( (x + t) \log(r)\big)}{r^{\abs{x}}} \ge \frac12 \min\{r^t, r^{-t}\}
 \]
 We therefore see that \((a_p)\) is a sequence such that \((r^{\abs{p}} a_p)\) is square integrable if and only if \(\cos_r D_s(a_p)\) is a sequence which is square integrable.
 Hence \(\cos_r D_s\) defines an isomorphism \(L_r^2 V \to L^2 V\).

 Thus \(\cos_r D_s\) defines a group isomorphism \(\gl(L_r^2 V) \cong \gl(L^2 V)\).
 We want to  show that this restricts to a group isomorphism \(\gl_\res(L_r^2 V) \cong \gl_\res(L^2 V)\).
 For the operators \(J_V\) and \(J_{V,r}\) defining \(\gl_\res(L^2 V)\) and \(\gl_\res(L^2_r V)\) respectively we can take the operators characterised by the fact that each takes \(v_j z^p\) to \(i \sign(p) v_j z^p\)\emhyp{}this makes sense for both \(L^2 V\) and \(L^2_r V\) as the set \(\{v_j z^p\}\) is a topological basis for both spaces (albeit only orthonormal for \(L^2 V\)).

As \(\cos_r D_s\) takes \(v_j z^p\) to a multiple of itself it is immediate that
\[
  (\cos_r D_s)^{-1} J_V (\cos_r D_s) = J_{V,r}.
\]
Thus for \(A \in \gl(L^2 V)\)
\[
  \lb (\cos_r D_s)^{-1} A (\cos_r D_s), J_{V,r} \rb = (\cos_r D_s)^{-1} \lb A, J_V \rb (\cos_r D_s).
\]
Hence \(\lb A, J_V \rb\) is Hilbert\enhyp{}Schmidt if and only if \(\lb (\cos_r D_s)^{-1} A (\cos_r D_s), J_{V,r} \rb\) is Hilbert\enhyp{}Schmidt.

We did not earlier describe the topology on \(\gl_\res\); in brief, it is given by combining the norm topology together with the Hilbert\enhyp{}Schmidt topology on the commutator \(\lb A, J \rb\).
From this and the above we see that the group isomorphism \(\gl_\res(L^2 V) \to \gl_\res(L^2_r V)\) is continuous.
\end{proof}

All of the above has a real counterpart given by taking the underlying real vector spaces.

\medskip

Although \(L^2_r V\) has an obvious inner product, we wish to allow for some variation in this choice.
For each choice of \(s\), as described in definition~\ref{def:cosds}, we define an inner product on \(L^2_r V\) by insisting that \(\cos_r D_s \colon L^2_r V \to L^2 V\) be an \emph{isometric} isomorphism.

A choice of \(s\), as in definition~\ref{def:cosds}, isometrically identifies the spaces \(L^2_r V\) with \(L^2(V)\) and thus defines a family of representations \(L_\pol U(V) \to \gl_\res(L^2 V)\).
We can extend this family to \(r = 1\) by including the standard action on \(L_\pol U(V)\) on \(L^2 V\).
It is straightforward to show that this family of representations is continuous in \(r\).
As an immediate consequence we have an explicit homotopy between the representation at some fixed \(r_0 \in (1, \infty)\) and the standard representation corresponding to \(r = 1\).

The main difference between the representations for \(r \in (1, \infty)\) and \(r = 1\) is that the representation of \(L_\pol U(V)\) is unitary for \(r = 1\) but not for any other \(r\).
If we are only concerned with its action on the Hilbert space this is not an issue since we are mainly interested in the induced \(\gl_\res\)\enhyp{}action which has a subgroup that acts unitarily.
Moreover, this subgroup is homotopic to \(\gl_\res\).
However, if we are also interested in the action on \(L V\) then we must look elsewhere as \(\gl_\res\) does not act on \(L V\).

\begin{defn}
Let \(c_\Z(\End(V))\) be the Banach space of \(\End(V)\)\enhyp{}valued, \Z\enhyp{}indexed sequences, \((A_k)\), with the property that \(\lim_{k \to \infty} A_k\) and \(\lim_{k \to -\infty} A_k\) exist and are equal.
Let \(E L \End(V)\) be the algebra of rapidly decreasing,  \Z\enhyp{}indexed sequences in \(c_\Z(\End(V))\); suitably topologised.
A sequence \((A_k)\) represents the operator
\[
  \sum A_k z^k
\]
and the product is defined accordingly.
Define the \emph{expanded loop group}, \(E L \gl(V)\), to be the group of units in this algebra with the standard topology making the two maps \(E L \gl(V) \to E L \End(V)\), \(g \mapsto g\) and \(g \mapsto g^{-1}\), continuous.
\end{defn}

This group has a variety of good properties, for example it is a split extension of \(L \gl(V)\).
The main one\emhyp{}for our purposes\emhyp{}is that for each \(r \in \lb 1,\infty \rp\) and \(s\) as in definition~\ref{def:cosds} there is a subgroup isomorphic to \(L_\pol U(V)\) which acts unitarily on \(L_r^2 V\).
Moreover, exactly as for \(\gl_\res\), the representations \(L_\pol U(V) \to E L \gl(V)\) so defined are all homotopic.

This group does not play a major r\^ole in this paper.
Its purpose is to show that it is possible to replace the ``locally equivalent'' co\hyp{}orthogonal structure with a ``locally trivial'' one by changing the structure group.
Therefore we shall not study it further.

%\newpage
\section{The Universal Case}
\label{sec:universal}

In this section we consider the universal case.
The classifying space of \(\Omega G\) is (homotopy equivalent to) \(G\).
Since \(\Omega G \simeq \Omega_\pol G\), this is also the classifying space of \(\Omega_\pol G\).
There is therefore a universal \(\Omega_\pol G\)\enhyp{}principal bundle over \(G\) such that the inclusion \(\Omega_\pol G \to \Omega G\) embeds this \(\Omega_\pol G\)\enhyp{}principal bundle in the standard \(\Omega G\)\enhyp{}principal bundle.

Using the \(\Omega_\pol G\)\enhyp{}bundle and \(\Omega G\)\enhyp{}bundle over \(G\) we obtain an \(L_\pol G\)\enhyp{}bundle and an \(L G\)\enhyp{}bundle in the obvious way.
Given an action of \(G\) on a vector space \(V\) we can therefore define vector bundles over \(G\) with fibres \(L V\), \(L_\pol V\), \(L_r^2 V\), and \(L^2 V\).

In this section we shall explicitly construct the \(L_\pol G\)\enhyp{}bundles over \(G\) for \(G\) each of \(U_n\), \(S U_n\), and \(S O_n\) and the associated vector bundles.
We start with some general results on polynomial loops.

\subsection{Polynomial Loops}
\label{sec:polloops}

In this part we consider the group of polynomial loops in a compact, connected Lie group.
This was studied extensively in~\cite{apgs} with some further work appearing in~\cite{gs2} in the case of \(U_n\).

Although the definition given in section~\ref{sec:polypol} of \(L_\pol G\) does not depend on the embedding of \(G\) in \(U_n\), it is useful to have such an embedding to investigate the structure of \(L_\pol G\) in a little more detail.
We consider loops of the form \(t \to \exp(t \xi)\) for suitable \(\xi \in \mf{g}\).
The main result is the following:

\begin{proposition}
 \label{prop:liepol}
 Let \(G\) be a compact, connected Lie group, \(\mf{g}\) its Lie algebra.
 For \(\xi \in \mf{g}\), let
  \(\eta_\xi \colon \R \to G\)
 denote the path
  \(\eta_\xi (t) = \exp(t \xi)\).

 Let
  \(\xi_1, \xi_2 \in \mf{g}\)
 be such that
  \(\exp(\xi_1) = \exp(\xi_2)\).
 Then
  \(\eta_{-\xi_1} \eta_{\xi_2}\)
 is a polynomial loop in \(G\).
\end{proposition}

As part of the proof of this, we shall prove the following useful result for the unitary group:

\begin{lemma}
 \label{lem:comlie}
 Let \(g \in U_n\).
 There exists
  \(\zeta \in \exp^{-1}(g) \subseteq \mf{u}_n\)
 such that \(\lb \zeta, \xi \rb = 0\) for all
  \(\xi \in \exp^{-1}(g)\).
\end{lemma}

The proofs of these rely on the simple structure in \(U_n\) of the centraliser of any particular element.
For \(g \in U_n\), define \(C(g)\) and \(Z(g)\) to be the centraliser of \(g\) and its centre.
That is,
 \(C(g) \coloneqq \{h \in G : h^{-1} g h = g\}\)
and \(Z(g) = Z(C(g))\).
Clearly, \(g \in Z(g)\).

\begin{lemma}
 For any \(g \in U_n\), \(Z(g)\) is a torus.
\end{lemma}

\begin{proof}
 The group \(C(g)\) is a closed subgroup of \(U_n\), hence its centre is a closed abelian subgroup of \(U_n\).
 In particular, it is compact.
 Therefore, it is a torus if and only if it is connected.

 Recall that two diagonalisable matrices commute if and only if they are simultaneously diagonalisable.
 This condition does not rely on the eigenvalues of either matrix but only on the eigenspaces.

 Let \(h \in Z(g)\).
 As \(h\) is unitary, it is orthogonally diagonalisable.
 Let
  \(\lambda_1, \dotsc, \lambda_l\)
 be the distinct eigenvalues of \(h\) with associated eigenspaces \(E_1, \dotsc, E_l\).
 For each \(j\), let
  \(s_j \in \lb - i \pi, i \pi \rp \)
 be such that \(e^{s_j} = \lambda_j\).

 Define
  \(\alpha \colon \lb 0, 1 \rb \to U_n\)
 to be the path such that \(\alpha(t)\) has eigenvalues \(e^{t s_j}\) and corresponding eigenspaces \(E_j\).
 Then \(\alpha(0) = 1_n\) and \(\alpha(1) = h\) so \(\alpha\) is a path from \(1_n\) to \(h\).
 By construction, \(\alpha(t)\) for \(t \ne 0\) has the same eigenspaces as \(h\) and therefore \(\alpha(t)\) commutes with exactly the same elements of \(U_n\) that \(h\) commutes with.
 Hence as \(h \in Z(g)\), \(\alpha(t) \in Z(g)\).
\end{proof}

\begin{proof}[Proof of lemma~\ref{lem:comlie}]
 As \(Z(g)\) is a torus, it is a connected compact Lie group.
 Therefore, the exponential map is surjective and so there is some
  \(\zeta \in \mf{z}(g) \subseteq \mf{u}_n\)
 with \(\exp(\zeta) = g\).
 As \(\zeta \in \mf{z}(g)\),
  \(\exp(t \zeta) \in Z(g)\)
 for all \(t \in \R\).

 Let \(\xi \in \mf{u}_n\) be such that \(\exp(\xi) = g\).
 Then for all \(t \in \R\), \(\exp(t \xi)\) commutes with \(g\).
 Hence \(\exp(t \xi) \in C(g)\)
 for all \(t\).
 Thus \(\exp(t \zeta)\) and \(\exp(t' \xi)\) commute for all \(t, t' \in \R\).
 Hence \(\lb \zeta, \xi \rb = 0\).
\end{proof}

Using this we can prove proposition~\ref{prop:liepol}.

\begin{proof}[Proof of proposition~\ref{prop:liepol}]
 Firstly, note that it is sufficient to prove this in the case of the unitary group.
 For if
  \(\eta_{-\xi_1} \eta_{\xi_2}\)
 is a loop in \(G\) which is a polynomial loop when \(G\) is considered as a subgroup of \(U_n\) then, by definition,
  \(\eta_{-\xi_1} \eta_{\xi_2}\)
 is a polynomial loop in \(G\).

 Secondly, note that it is sufficient to consider the case where \(\xi_2 = 0\).
 This forces \(\exp(\xi_1) = 1_n\).
 To deduce the general case from this simpler one, note that by lemma~\ref{lem:comlie} that there is some \(\zeta \in \mf{u}_n\) with
  \(\exp(\zeta) = \exp(\xi_1)\)
 (whence also \(\exp(\xi_2)\)) such that \(\lb \zeta,\xi_j \rb = 0\).
 Then
  \(\exp(\xi_j - \zeta) = 1_n\)
 so, by assumption,
  \(\eta_{(\xi_j - \zeta)}\)
 is a polynomial loop.
 The identity:
 \[
  \eta_{-\xi_1} \eta_{\xi_2} = \eta_{-\xi_1} \eta_\zeta \eta_{-\zeta} \eta_{\xi_2} = \eta_{(-\xi_1 + \zeta)} \eta_{(\zeta - \xi_2)}.
 \]
 demonstrates that this is a polynomial loop.

 Thus we need to show that \(\eta_\xi\) is a polynomial loop if \(\exp(\xi) = 1\).
 To show this, we diagonalise \(\xi\).
 If \(s\) is an eigenvalue of \(\xi\) then \(e^{s}\) is an eigenvalue of \(\exp(\xi) = 1\).
 The eigenvalues of \(\xi\) therefore lie in \(2 \pi i \Z\).
 Hence there is a basis of \(\C^n\) with respect to which \(\eta_\xi\) is the path:
 \[
  t \to
  \begin{bmatrix}
   e^{2 \pi i t k_1} & 0 &
   \dots &
   0 \\
   0 &
   e^{2 \pi i t k_2} &
   \dots &
   0 \\
   \hdotsfor{4} \\
   0 &
   0 &
   \dots &
   e^{2 \pi i t k_n}
  \end{bmatrix}
 \]
 for some \(k_j \in \Z\).
 Since
  \(e^{2 \pi i t k} = z^k\)
 for \(k \in \Z\), this is a polynomial loop (viewed as a periodic path).
\end{proof}

Note that for a general group \(G\), although the loop
 \(\eta_{\xi_1} \eta_{-\xi_2}\)
lies in \(\Omega_\pol G\), there may be no factorisation in \(G\) as
 \(\eta_{\xi_1 - \zeta} \eta_{\zeta - \xi_2}\)
since Lemma~\ref{lem:comlie} need not hold for a general Lie group.

\subsection{The Path Spaces}
\label{sec:path}

In this part we shall give an explicit construction of the principal \(L_\pol G\)\enhyp{}bundle over \(G\) for \(G\) each of \(U_n\), \(S U_n\), and \(S O_n\).
We shall also construct a similar bundle for the smooth loop group.
These bundles will be denoted by \(P_\pol G\) and \(P_\per G\) (the ``\(\per\)'' stands for ``periodic'').

To demonstrate that these are principal bundles with the appropriate fibre we have to show two things: firstly, that the bundles are locally trivial; and secondly, that the fibres have an action of the appropriate loop group which identifies the fibre with that group.
The second of these is straightforward, the first is simple for the smooth case but is surprisingly difficult for the polynomial loop group.
We shall only consider the cases of \(U_n\), \(S U_n\), and \(S O_n\).

\begin{defn}
 \label{def:perpol}
 Let \(G\) be a compact, connected Lie group, \(\mf{g}\) its Lie algebra.
 We define \(P_\per G\) and \(P_\pol G\) as follows:
 \begin{enumerate}
 \item
   \(P_\per G\) is the space of smooth paths
   \(\alpha \colon \R \to G\)
  with the property that
   \(\alpha(t + 1)\alpha(t)^{-1}\)
  is constant.

 \item
   \(P_\pol G \subseteq P_\per G\)
  consists of those paths of the form \(\eta_\xi \gamma\) for some \(\xi \in \mf{g}\) and
   \(\gamma \in L_\pol G\).
 \end{enumerate}

 The projection map \(P_\per G \to G\) is given by
  \(\alpha \to \alpha(1) \alpha(0)^{-1}\).
 Notice that when restricted to \(P_\pol G\), this maps \(\eta_\xi \gamma\) to \(\exp(\xi)\).
\end{defn}

Recall from section~\ref{sec:polloops} that for \(\xi \in \mf{g}\) the path
 \(\eta_\xi \colon \R \to G\)
is defined as the path \(t \to \exp( t \xi)\).

Observe that a path in \(P_\per G\) is completely determined by its values on the interval \(\lb 0,1 \rb\).
However, not every smooth path \(\lb 0,1 \rb \to G\) defines an element of \(P_\per G\): one needs certain conditions on the derivatives at the endpoints.
Defining \(P_\per G\) as we have seems the simplest way to state these conditions.

It will sometimes be useful to consider an element of \(P_\per G\) to be a pair
 \((g, \alpha) \in G \times P G\)
such that
 \(\alpha(t + 1) = g\alpha(t)\).
Here \(P G\) is \emph{all} smooth paths \(\R \to G\).
Although \(g\) is completely determined by \(\alpha\), this viewpoint makes it more explicit.

We shall now investigate the desired properties of these spaces.
Neither is a group (unlike the analogous continuous situation), but the group \(G\) acts in two ways:

\begin{lemma}
 \label{lem:pathconj}
 The group \(G\) acts on \(P_\per G\) by two actions:
 \[
  g \cdot_m \alpha = g \alpha, \qquad g \cdot_c \alpha = g \alpha g^{-1}.
 \]
 These actions restrict to actions on \(P_\pol G\).
 For both actions, the action of \(G\) on itself by conjugation makes the projection \(P_\per G \to G\) \(G\)\enhyp{}equivariant (hence also for \(P_\pol G \to G\)).
\end{lemma}

\begin{proof}
 Let \(g \in G\) and \(\alpha \in P_\per G\).
 Both \(g \alpha\) and \(g \alpha g^{-1}\) are smooth paths in \(G\) so we only need to check the periodicity condition.
 Let
  \(h = \alpha(t + 1) \alpha(t)^{-1}\).
 Then:
 \begin{align*}
  (g \alpha)(t + 1) (g \alpha)(t)^{-1} &%
  = g \alpha(t + 1) \alpha(t)^{-1} g^{-1} = g h g^{-1}.
  \\
  (g \alpha g^{-1}) (t + 1) (g \alpha g^{-1})(t)^{-1} &%
  = g \alpha(t + 1) g^{-1} g \alpha(t)^{-1} g^{-1} \\
  &%
  = g \alpha(t + 1) \alpha(t)^{-1} g^{-1} = g h g^{-1}.
 \end{align*}
 This also proves the statement about the induced action on \(G\).

 If \(\alpha \in P_\pol G\) then \(\alpha\) is of the form \(\eta_\xi \gamma\) for some \(\xi \in \mf{g}\) and \(\gamma \in L_\pol G\).
 Let \(h\) be either \(g^{-1}\) or \(1_G\).
 Then
  \(g \eta_\xi \gamma h = \eta_{(\Ad_g \xi)} g \gamma h\).
 As \(L_\pol G\) is closed under left and right multiplication by \(G\), this lies in \(P_\pol G\) as required.
\end{proof}

\begin{proposition}
 Define an action of \(L G\) on \(P_\per G\) by sending
  \((\alpha, \gamma) \in P_\per G \times L G\)
 to the path
  \(t \to \alpha(t) \gamma(t)\).
 This action is well\hyp{}defined and identifies the fibres of \(P_\per G \to G\) with \(L G\).
 It restricts to an action of \(L_\pol G\) on \(P_\pol G\) and identifies the fibres of \(P_\pol G \to G\) with \(L_\pol G\).
\end{proposition}

\begin{proof}
 The path
  \(t \to \alpha(t) \gamma(t)\)
 is a smooth path \(\R \to G\) (considering \(\gamma\) as a periodic path).
 We need merely check the periodicity condition.
 Since
  \(\gamma(t + 1) = \gamma(t)\)
 for all \(t \in \R\), we have:
 \begin{align*}
  (\alpha \gamma) (t + 1) (\alpha \gamma)(t)^{-1} &%
  = \alpha(t + 1) \gamma (t + 1) \gamma(t)^{-1} \alpha(t)^{-1} \\
  &%
  = \alpha(t + 1) \alpha(t)^{-1}.
 \end{align*}
 Hence
  \(\alpha \gamma \in P_\per G\).
 This also shows that \(\alpha \gamma\) lies in the same fibre as \(\alpha\).

 For an inverse, let
  \(\alpha, \beta \in P_\per G\)
 be such that
  \(\alpha(1) \alpha(0)^{-1} = \beta(1) \beta(0)^{-1}\).
 As \(\alpha\) and \(\beta\) lie in \(P_\per G\), this means that
  \(\alpha(t + 1)\alpha(t)^{-1} = \beta(t + 1) \beta(t)^{-1}\)
 for all \(t \in \R\).
 Rearranging this yields
  \(\alpha(t + 1)^{-1} \beta(t + 1) = \alpha(t)^{-1} \beta(t)\).
 Thus the path \(\gamma\) given by
  \(\gamma(t) = \alpha(t)^{-1} \beta(t)\)
 is a loop.
 Moreover, it is smooth.
 Clearly
  \(\alpha \gamma = \beta\)
 so this is the inverse map which identifies a non\hyp{}empty fibre of \(P_\per G \to G\) with \(L G\).

 In the polynomial case, if \(\alpha \in P_\pol G\) and
  \(\gamma \in L_\pol G\)
 then by definition,
  \(\alpha = \eta_\xi \beta\)
 for some polynomial loop \(\beta\).
 Therefore
  \(\alpha \gamma = \eta_\xi (\beta \gamma)\)
 and hence lies in \(P_\pol G\).

 Conversely, suppose that
  \(\alpha, \beta \in P_\pol G\)
 lie in the same fibre.
 We need to show that the loop
  \(t \to \alpha^{-1}(t) \beta(t)\)
 is a polynomial loop.
 Let
  \(\alpha = \eta_{\xi_1} \widehat{\alpha}\)
 and
  \(\beta = \eta_{\xi_2} \widehat{\beta}\)
 where \(\widehat{\alpha}\) and \(\widehat{\beta}\) are polynomial loops.
 Since \(\alpha\) and \(\beta\) lie in the same fibre,
  \(\exp(\xi_1) = \exp(\xi_2)\).
 Thus:
 \[
  \gamma = \widehat{\alpha}^{-1} \eta_{-\xi_1} \eta_{\xi_2} \widehat{\beta}.
 \]
 By proposition~\ref{prop:liepol}, the two terms in the centre give a polynomial loop, hence \(\gamma\) is a polynomial loop.

 To complete the proof of the proposition, we need to show that no fibre of \(P_\pol G \to G\) is empty, whence also no fibre of \(P_\per G \to G\) is empty.
 As \(G\) is a compact, connected Lie group, for each \(g \in G\) there is some \(\xi \in \mf{g}\) such that \(\exp(\xi) = g\).
 The path \(\eta_\xi\) lies in \(P_\pol G\) (and thus in \(P_\per G\)) and is in the fibre above \(g\).
 Thus the fibres are non\hyp{}empty.
\end{proof}

Proving that \(P_\per G\) is locally trivial is relatively straightforward.
The case of \(P_\pol G\) is harder.
Therefore we deal with \(P_\per G\) quickly now before passing to the\emhyp{}for this paper\emhyp{}more relevant case of the polynomial loops in the next section.

\begin{proposition}
 The space \(P_\per G\) is locally trivial over \(G\).
\end{proposition}

\begin{proof}
 To prove this, we require local sections.
 Let \(g \in G\).
 Let \(\xi \in \mf{g}\) be such that \(\exp(\xi) = g\).
 Let
  \(\rho \colon \lb 0,\frac12 \rb \to \lb 0, \frac12 \rb\)
 be a smooth surjection which preserves the endpoints and is constant in a neighbourhood of each endpoint.
 Let \(\phi \colon V \to U\) be a chart for \(G\) with \(U\) a neighbourhood of \(g\) such that \(\phi^{-1}(g) = 0\).

 For \(h \in U\), define a path
  \(\alpha_h \colon \lb 0, 1 \rb \to G\)
 by:
 \[
  \alpha_h(t) = \begin{cases} \exp(2\rho(t) \xi) &
   t \in \lb 0, \frac12 \rb \\
   \phi((2\rho(t - \frac12) + 1)\phi^{-1}(h)) &
   t \in \lb \frac12, 1 \rb
   	 \end{cases}
 \]
 By construction, \(\alpha_h\) is continuous.
 Since \(\alpha_h\) is constant in a neighbourhood of \(\frac12\) and is smooth either side, it is smooth.
 Moreover, as it is constant in neighbourhoods of \(0\) and \(1\), the concatenation
  \(\alpha_h \sharp (\alpha_h(1) \alpha_h)\)
 is smooth.
 Hence \(\alpha_h\) extends via the formula:
 \[
  \alpha_h(t + n) = \alpha_h(1)^n \alpha_h(t)
 \]
 for \(t \in \lb 0,1 \rp\) and \(n \in \Z\), to a smooth path \(\R \to G\) such that
  \(\alpha_h(t + 1) = \alpha_h(1) \alpha_h(t)\)
 for all \(t \in \R\).

 Clearly, \(\alpha_h(1) = h\).
 Also, the assignment \(h \to \alpha_h\) is smooth.
 Therefore, \(h \to \alpha_h\) is a local section of \(P_\per G\) in a neighbourhood of \(g\).
\end{proof}

\subsubsection{The Polynomial Path Space}

The case of the polynomial path space is harder.
Regarding local sections, it would appear from the definition that there are natural local sections, namely \(g \to \eta_\xi\) where \(\exp(\xi) = g\).
However, except in the case of the unitary group, there is in general no way to choose \(\xi\) smoothly in \(g\) for all points \(g \in G\) (it is always possible to do so for an open dense subset, but this is not good enough).

In fact, we are not able to prove that \(P_\pol G \to G\) is locally trivial for all compact, connected \(G\) at this time.
The methods we employ work on a case\hyp{}by\hyp{}case basis.
This is sufficient for our needs as we are mainly interested in ordinary vector bundles with inner products and thus in the structure groups \(U_n\) and \(S O_n\).
We shall prove that \(P_\pol G \to G\) is locally trivial for these groups and also for \(S U_n\).
There is no \emph{a priori} reason why the argument for \(S O_n\) should not extend to \(S p_n\), using quaternionic structures in place of complex structures but we feel that this case is outside the focus of this paper.

The following result will prove useful in examining the structure of \(P_\pol G\) in terms of \(P_\pol U_n\).

\begin{lemma}
 \label{lem:polsub}
 Let \(G\) be a compact, connected Lie group.
 Consider \(G\) as a subgroup of some \(U_n\).
 Then
  \(P_\pol G = P_\per G \cap P_\pol U_n\).
\end{lemma}

\begin{proof}
 Clearly
  \(P_\pol G \subseteq P_\per G \cap P_\pol U_n\).
 For the converse, let
  \(\alpha \in P_\per G \cap P_\pol U_n\).
 Then
  \(\alpha = \eta_\xi \gamma\)
 for some \(\xi \in \mf{u}_n\) and
  \(\gamma \in L_\pol U_n\).
 Now
  \(\exp(\xi) = \eta_\xi(1) = \alpha(1) \in G\)
 since \(\alpha \in P_\per G\).
 Choose \(\zeta \in \mf{g}\) such that
  \(\exp(\zeta) = \exp(\xi)\).
 Then:
 \[
  \alpha = \eta_\zeta \eta_{-\zeta} \eta_\xi \gamma.
 \]
 By proposition~\ref{prop:liepol},
  \(\eta_{-\zeta} \eta_\xi\)
 is a polynomial loop in \(U_n\).
 Since \(\alpha\) and \(\eta_\zeta\) both take values in \(G\),
  \(\eta_{-\zeta} \eta_\xi \gamma\)
 must also take values in \(G\).
 It thus lies in \(L G \cap L_\pol U_n\) which is, by definition, \(L_\pol G\).
 Therefore \(\alpha\) is of the form \(\eta_\zeta \beta\) with \(\zeta \in \mf{g}\) and \(\beta \in L_\pol G\).
 Hence \(\alpha \in P_\pol G\).
\end{proof}

\subsubsection{The Unitary Group}

In the case of \(U_n\), there are local sections of the form \(g \to \eta_\xi\) where \(\exp(\xi) = g\).
This will follow from lemma~\ref{lem:comlie}.

\begin{proposition}
 \label{prop:unloctriv}
 The space \(P_\pol U_n\) is locally trivial over \(U_n\).
\end{proposition}

\begin{proof}
 Let \(s \in i \R\).
 Let \(V_s \subseteq U_n\) be the open subset consisting of those operators which do not have \(-e^s\) as an eigenvalue.
 Let
  \(\mf{v}_s \subseteq \mf{u}_n\)
 be the open subset consisting of those operators which have eigenvalue in the interval
  \((s - i \pi, s + i \pi)\).
 The exponential map restricts to a diffeomorphism
  \(\exp \colon \mf{v}_s \to V_s\).
 Let
  \(\log_s \colon V_s \to \mf{v}_s\)
 be its inverse.

 For a direct construction, define the \(s\)\enhyp{}logarithm
  \(\log_s \colon \T \ssetminus \{-e^s\} \to (s - i \pi, s + i \pi)\)
 as the inverse of the exponential map on this domain (note that this coincides with the above definition putting \(n = 1\)).
 Let \(g \in V_s\).
 Let
  \(E_1 \oplus \dotsb \oplus E_l\)
 be the orthogonal decomposition of \(\C^n\) into the eigenspaces of \(g\) with eigenvalues
  \(\lambda_1, \dotsc, \lambda_l\).
 Then \(\log_s g\) is the operator which acts on \(E_j\) by multiplication by \(\log_s \lambda_j\).

 It is a simple exercise to show that \(\log_s g \in Z(g)\) for any \(g\) and \(s\) such that \(\log_s g\) is defined, that \(\log_s g\) is locally constant in \(s\), and that
  \(V_{s + 2 \pi i} = V_s\)
 and
  \(\mf{v}_{s + 2 \pi i} = \mf{v}_{s} + 2 \pi i 1_n\).

 The local sections of \(P_\pol U_n \to U_n\) are
  \(\alpha_s \colon V_s \to P_\pol U_n\)
 given by
  \(\alpha_s(g)(t) = \exp(t \log_s g)\).
\end{proof}

\subsubsection{The Special Unitary Group}

The method of the previous section works in \(U_n\) because every point in \(U_n\) is \emph{exp\hyp{}regular}; that is, is the image of a point in \(\mf{u}_n\) such that the exponential map is a diffeomorphism is a neighbourhood of that point.
This is not true for a general Lie group.
It is straightforward to show that the preimage of \(-1 \in S U_2\) under
 \(\exp \colon \mf{s u}_2 \to S U_2\)
is a countable number of copies of \(\CP^1\), hence \(-1 \in S U_2\) is not exp\hyp{}regular.

However, we can still prove that
 \(P_\pol S U_n \to S U_n\)
is locally trivial.
The strategy is to use the fact that there \emph{is} a point in \(\mf{u}_n\) around which the exponential map is a local diffeomorphism, and then use the fact that
 \(S U_n \to U_n \to S^1\)
is split.

\begin{proposition}
 The map
  \(P_\pol S U_n \to S U_n\)
 is locally trivial.
\end{proposition}

\begin{proof}
 Choose a unit vector \(v \in \C^n\).
 Define the representation
  \(\sigma \colon \T \to U_n\)
 by
  \(\sigma(\lambda)v = \lambda v\)
 and \(\sigma(\lambda)\) is the identity on
  \(\langle v \rangle^\bot\).

 Let \(s \in i \R\).
 Let \(V_s \subseteq U_n\) and
  \(\mf{v}_s \subseteq \mf{u}_n\)
 be as in the proof of proposition~\ref{prop:unloctriv}.
 Let
  \(\alpha_s \colon V_s \to P_\pol U_n\)
 be the local section defined in that proposition.

 Define
  \(\beta_s \colon V_s \cap S U_n \to P_\per U_n\)
 by:
 \[
  \beta_s(g)(t) = \alpha_s(g)(t) \, \sigma\Big(\det\big(\alpha_s(g)(-t)\big)\Big).
 \]

 Recall that
  \(\det \exp(\xi) = e^{\tr \xi}\).
 Thus for
  \(g \in V_S \cap S U_n\):
 \[
  \det \alpha_s(g)(-t) = e^{\tr( -t \log_s(g))} = e^{-t \tr \log_s(g)}.
 \]
 As \(g \in S U_n\),
  \(e^{\tr \log_s(g)} = \det g = 1\)
 so
  \(\tr \log_s(g) = 2 \pi i k\)
 for some \(k \in \Z\).
 Thus
  \(t \to e^{-t \tr \log_s(g)}\)
 is the map \(t \to z^{-k}\).
 Hence
  \(\sigma( \det \alpha_s(g)(-t))\)
 is a polynomial loop in \(U_n\).
 Thus
  \(\beta_s(g)(t) \in P_\pol U_n\).

 Then as
  \(\det \circ \sigma \colon \T \to \T\)
 is the identity,
  \(\det \beta_s(g)(t) = 1\)
 for all \(g, t\).
 Hence
  \(\beta_s(g)(t) \in S U_n\)
 for all \(g, t\).
 Thus by lemma~\ref{lem:polsub},
  \(\beta_s(g) \in P_\per S U_n \cap P_\pol U_n = P_\pol S U_n\).
\end{proof}

\subsubsection{The Special Orthogonal Group}
\label{sec:son}

The situation for \(S O_n\) is more complicated still.
The problem here is with eigenvalue \(-1\).
It can be shown that \(g \in S O_n\) is \emph{exp\hyp{}regular} if and only if its \(-1\)\enhyp{}eigenspace has dimension at most \(2\).
The solution comes from the theory of \emph{unitary structures} which we now describe.

\begin{defn}
 Let \(E\) be a real vector space with an inner product.
 A \emph{unitary structure} on \(E\) is an orthogonal map \(J \colon E \to E\) such that \(J^2 = -1\).
\end{defn}

\begin{proposition}
 \label{prop:cplxstr}
 Let \(E\) be a real even dimensional vector space with an inner product.
 The properties of unitary structures that we shall need are:

 \begin{enumerate}
 \item
   \(E\) admits a unitary structure.

  \label{it:cplxeven}

 \item
  The set of unitary structures on \(E\) is
   \(O(E) \cap \mf{o}(E)\) (in \(\End(E)\)).

  \label{it:cplxorth}

 \item
  Let \(J\) be a unitary structure on \(E\).
  Then
   \(\exp(\pi J) = -1_E\).

  \label{it:cplxexp}

 \item
  Let \(J_1, J_2\) be unitary structures on \(E\).
  Then:
   \( \eta_{-\pi J_1} \eta_{\pi J_2} \)
  is a polynomial loop in \(S O(E)\).

  \label{it:cplxpol}

 \item
  Let \(\xi \in \mf{s o}(E)\) be such that \(\xi\) does not have \(0\) as an eigenvalue.
  Then there is a natural unitary structure \(J_\xi\) on \(E\) which varies smoothly in \(\xi\).
  Considered as an element of \(\mf{s o}(E)\), \(J_\xi\) satisfies \(\lb \xi, J_\xi \rb = 0\).
  The assignment \(\xi \to J_\xi\) satisfies \(J_J = J\) (here \(J\) is considered as an element of \(\mf{s o}(E)\)), and
   \(J_{\xi + c J_\xi} = J_\xi\)
  for \(c > 0\).

  \label{it:cplxalg}

 \item
  Let \(g \in S O(E)\) be such that \(1\) is not an eigenvalue of \(g\).
  Then \(\log_0(-g)\) is of the form \(\xi - \pi J_\xi\) for some \(\xi \in \mf{s o}(E)\) with \(\exp(\xi) = g\).

  \label{it:cplxdecomp}
 \end{enumerate}
\end{proposition}

In the last property we use the inclusion
 \(S O(E) \to U(E \otimes \C)\)
to define
 \(\log_0 \colon S O(E) \cap V_0 \to \mf{u}(E)\).
Since \(\log_0\) commutes with complex conjugation%
\footnote{%
 It is the only one of the logarithms that we have defined with this property.%
}%
, the image of \(S O(E) \cap V_0\) lies in \(\mf{s o}(E)\).

\begin{proof}
 Property~\ref{it:cplxeven} is a standard property of complex structures whilst~\ref{it:cplxorth} is a simple deduction from the definition of a unitary structure.
 Therefore we start with property~\ref{it:cplxexp}.

 \begin{enumerate}
  \addtocounter{enumi}{2}
 \item
  As an element of
   \(\mf{o}(E) = \mf{s o}(E)\),
   \(J\) is diagonalisable over \C. Since \(J^2 = -1\), its eigenvalues are \(\pm i\).
  Thus \(\pi J\) has eigenvalues \(\pm \pi i\).
  Hence \(\exp(\pi J)\) has sole eigenvalue \(-1\).
  As
   \(\exp(\pi J) \in S O(E)\),
  it is diagonalisable over \C and thus is \(-1_E\).

 \item
  This is a corollary of proposition~\ref{prop:liepol} together with the previous property.

 \item
  Diagonalise \(\xi\) over \C. As \(\xi\) is a real operator, its eigenvalues and corresponding eigenspaces come in conjugate pairs.
  As \(\xi\) is skew\hyp{}adjoint, its eigenvalues lie on the imaginary axis in \C. Let
   \(W \subseteq E \otimes \C\)
  be the sum of the eigenspaces of \(\xi\) corresponding to eigenvalues of the form \(i s\) with \(s > 0\).
  Then \(\conj{W}\), resp.\ \(W^\bot\), is the sum of the eigenspaces of \(\xi\) corresponding to eigenvalues of the form \(i s\) with \(s < 0\), resp.\ \(s \le 0\).
  The assumption on \(\xi\) implies that
   \(\conj{W} = W^\bot\).
  Define \(J_\xi\) on \(E \otimes \C\) to be the operator with eigenspaces \(W\) and \(\conj{W}\) with respective eigenvalues \(i\) and \(-i\).
  By construction, \(J^2 = -1\) and \(J^* J = 1\).
  As the eigenspaces and eigenvalues of \(J\) come in conjugate pairs, \(J\) is a real operator and thus is a unitary structure.

  Since \(J_\xi\) is defined from the eigenspaces of \(\xi\), it varies smoothly in \(\xi\).
  Moreover, as the eigenspaces of \(J_\xi\) decompose as eigenspaces of \(\xi\), \(J_\xi\) and \(\xi\) are simultaneously diagonalisable over \C. Hence \(\lb \xi, J_\xi \rb = 0\).

  It is clear from the construction that if \(\zeta\) and \(\xi\) can be simultaneously diagonalised and the eigenvalues of \(\zeta\) have the same parity on the imaginary axis as the corresponding ones of \(\xi\) then \(J_\zeta = J_\xi\).
  In particular, \(J_J = J\) and
   \(J_{\xi + c J_\xi} = J_\xi\)
  for \(c > 0\).

 \item
  Let \(F\) be the \(-1\)\enhyp{}eigenspace of \(g\).
  Then \(E\) decomposes \(g\)\enhyp{}invariantly as \(F \oplus F^\bot\).
  As \(g\) does not have \(1\) as an eigenvalue, \(\log_0(-g)\) is well\hyp{}defined.
  Since the decomposition of \(E\) is \(-g\)\enhyp{}invariant:
  \[
   \log_0(-g) = \log_0(-g \restrict_{F}) + \log_0(-g \restrict_{F^\bot}) = \log_0(-g \restrict_{F^\bot}).
  \]
  This last step is because \(-g \restrict_F = 1_F\)
  so
   \(\log_0(-g \restrict_F) = 0_F\).

  Let
   \(\xi_{F^\bot} = \log_0(-g \restrict_{F^\bot})\).
  As \(-g\) does not have \(1\) as an eigenvalue on \(F^\bot\), \(\xi_{F^\bot}\) does not have \(0\) as an eigenvalue.
  Let \(J_{F^\bot}\) be the corresponding unitary structure.
  As
   \(\lb \xi_{F^\bot}, J_{F^\bot} \rb = 0\),
  \[
   \exp(\xi_{F^\bot} + \pi J_{F^\bot}) = \exp(\xi_{F^\bot}) \exp( \pi J_{F^\bot}) = (-g) \restrict_{F^\bot} (-1_{F^\bot}) = g \restrict_{F^\bot}.
  \]

  As \(g \in S O(E)\), \(F\) must be of even dimension.
  Choose a unitary structure \(J_F\) on \(F\).
  Then
   \(\exp(\pi J_F) = -1_F = g \restrict_F\).
  Let
   \(\xi = \pi J_F + \xi_{F^\bot} + \pi J_{F^\bot}\).
  Then:
  \[
   \exp(\xi) = \exp(\pi J_F) + \exp(\xi_{F^\bot} + \pi J_{F^\bot}) = -1_F + g \restrict_{F^\bot} = g.
  \]
  Then
   \(J_\xi = J_F + J_{F^\bot}\)
  so
   \(\xi - \pi J_\xi = \xi_{F^\bot}\),
  whence
   \(\xi - \pi J_\xi = \log_0 (-g)\).\qedhere
 \end{enumerate}
\end{proof}

\begin{theorem}
 \label{th:lctrso}
 The map
  \(P_\pol S O_n \to S O_n\)
 is locally trivial.
\end{theorem}

\begin{proof}
 We first describe a family of open sets which cover \(S O_n\).
 These will be the domains of the sections of \(P_\pol S O_n\).
 The family is indexed by the interval \(\lb -1,1 \rb\) and by elements of \(S O_n\).

 Let \(r \in \lb -1,1 \rb\).
 Let \(W_r\) be the open subset of \(S O_n\) consisting of those \(g\) such that no eigenvalue of \(g\) (over \C) has real part \(r\).
 For \(g \in W_r\) there is a \(g\)\enhyp{}invariant orthogonal decomposition of \(\R^n\) as
  \(E_{-1}^r(g) \oplus E_r^1(g)\)
 where the eigenvalues (over \C) of \(g\) on \(E_{-1}^r(g)\) have real part in the interval \(\lb -1,r \rb\) and on \(E_r^1(g)\) in the interval \(\lb r, 1 \rb\).
 Note that \(g\) cannot have eigenvalue \(1\) on \(E_{-1}^r(g)\), even if \(r = 1\), so as \(g \in S O_n\), \(E_{-1}^r(g)\) must have even dimension.

 Over each \(W_r\) is a vector bundle with fibre \(E_{-1}^r(g)\) at \(g\) (this will have different dimension on the different components of \(W_r\)).
 Over most \(W_r\)'s this bundle is not trivial.
 Therefore we find smaller open sets over which we can trivialise it.

 Let \(r \in \lb -1,1 \rb\) and \(g \in W_r\).
 Define \(W_r(g)\) to be the open subset of \(S O_n\) consisting of those \(h \in W_r\) for which the orthogonal projection
  \(E_{-1}^r(h) \to E_{-1}^r(g)\)
 is an isomorphism.

 Over \(W_r(g)\), therefore, the aforementioned vector bundle is trivial and of constant even dimension.
 Hence, we can choose a unitary structure \(J_h\) on each \(E_{-1}^r(h)\) which varies smoothly in \(h\).

 Extend \(J_h\) to a skew\hyp{}adjoint operator on \(\R^n\) by defining it to be zero on \(E_r^1(h)\).
 Let
  \(\epsilon(h) = h \exp(-\pi J_h) \in S O_n\).
 Then \(\epsilon(h)\) agrees with \(h\) on \(E_r^1(h)\) and is \(-h\) on \(E_{-1}^r(h)\).
 Since \(h\) does not have eigenvalue \(-1\) on \(E_r^1(h)\) and does not have eigenvalue \(1\) on \(E_{-1}^r(h)\), \(\epsilon(h)\) does not have eigenvalue \(-1\) on \(\R^n\) and so lies in the domain of \(\log_0\).
 Also, as \(J_h\) varies smoothly in \(h\), \(h \to \epsilon(h)\) is smooth.

 Define
  \(\beta_{r,g} \colon W_r(g) \to P_\per S O_n\)
 by:
 \[
  \beta_{r,g}(h)(t) = \exp\big( t \log_0( \epsilon(h)) \big) \exp(t \pi J_h).
 \]
 This is a smooth path in \(S O_n\) since both
  \(\log_0 ( \epsilon(h))\)
 and \(J_h\) lie in \(\mf{s o}_n\).
 It varies smoothly in \(h\) since both \(\epsilon(h)\) and \(J_h\) are smooth in \(h\).
 Since
  \(\epsilon(h) = h \exp(-\pi J_h)\),
  \(\beta_{r,g}(h)(1) = h\)
 so it is a path above \(h\).
 We need to show that it lies in \(P_\pol S O_n\).

 Now \(\epsilon(h)\) respects the decomposition
  \(E_{-1}^r(h) \oplus E_r^1(h)\)
 of \(\R^n\), therefore so does \(\log_0(\epsilon(h))\).
 Accordingly, write
  \(\log_0(\epsilon(h)) = \xi_{-1}^r + \xi_r^1\).

 Consider the situation on \(E_{-1}^r(h)\).
 Since
  \(\exp(\xi_{-1}^r) = \epsilon(h) = -h\)
 (all restricted to \(E_{-1}^r(h)\)), by property~\ref{it:cplxdecomp},
  \(\xi_{-1}^r = \zeta - \pi J_\zeta\)
 for some
  \(\zeta \in \mf{s o}(E_{-1}^r(h))\)
 with \(\exp(\zeta) = h\).
 Extend \(J_\zeta\) to \(\R^n\) by defining it to be zero on \(E_r^1(h)\).
 Let
  \(\xi = \zeta + \xi_r^1\).
 Then \(\exp(\xi) = h\), \(\lb \xi, J_\zeta \rb = 0\), and
  \(\log_0( \epsilon(h)) = \xi - \pi J_\zeta\).
 Therefore:
 \[
  \beta_{r,g}(h)(t) = \exp(t \xi) \exp(-t \pi J_\zeta) \exp(t \pi J_h).
 \]

 Now \(J_\zeta\) and \(J_h\) are both extensions to \(\R^n\) by zero of unitary structures on \(E_{-1}^r(h)\), so by property~\ref{it:cplxpol},
  \(\exp(-t \pi J_\zeta) \exp(t \pi J_h)\)
 is a polynomial loop in \(S O_n\).
 Hence \(\beta_{r,g}(h)\) lies in \(P_\pol S O_n\).
\end{proof}

\subsection{The Polynomial Vector Bundles}
\label{sec:polvect}

Now that we have principal bundles, given a representation we can construct vector bundles.
Let \(V\) be a finite dimensional vector space with an inner product, either real or complex.
Let \(L V\) be the space of smooth loops in \(V\) and \(L_\pol V\) the space of polynomial loops.
If \(V\) is complex then
 \(L_\pol V = V \lb z^{-1},z \rb\);
if \(V\) is real then
 \(L_\pol V = L V \cap L_\pol (V \otimes \C)\).

Let \(G\) be a compact, connected Lie group which acts on \(V\) by isometries.
In the polynomial case, assume that \(G\) is one of \(U_n\), \(S U_n\), or \(S O_n\).
Then \(L G\) acts on \(L V\) and \(L_\pol G\) acts on \(L_\pol V\).
Therefore we have vector bundles over \(G\) together with a bundle inclusion:
\[
 P_\pol V \coloneqq P_\pol G \times_{L_\pol G} L_\pol V \to P_\per V \coloneqq P_\per G \times_{L G} L V.
\]

We shall now give an alternative view of these vector bundles which will be more enlightening in terms of their structure.

Let \(P V\) be the full path space of \(V\).
Define
 \(\tau \colon P V \to P V\)
to be the shift operator:
 \((\tau \beta )(t) = \beta(t + 1)\).
Let \(D\) denote the differential operator:
 \((D \beta)(t) = \diff{\beta}{t}(t)\).
There is a strong connection between these operators: \(D\) is the infinitesimal generator of the group of translations on \(P V\) and \(\exp(D) = \tau\).

The motivation for considering these operators is that they give simple descriptions of \(L V\) and \(L_\pol V\) inside \(P V\).
The loop space, \(L V\), is the \(+1\)\enhyp{}eigenspace of \(\tau\).
The space of polynomial loops inside \(L V\) is the union of the finite dimensional \(D\)\enhyp{}invariant subspaces of \(L V\).

In the complex case, we can write this as the linear span of the eigenvectors of \(D\).
This does not carry over to the real case, however, as the only eigenvectors of \(D\) are the constant maps.

\begin{theorem}
 Let \(g\) in \(G\).
 The fibre of \(P_\per V\) above \(g\) is the space of \(\phi \in P V\) such that \(\tau \phi = g \phi\).

 The fibre of \(P_\pol V\) above \(g\) is the union of the finite dimensional \(D\)\enhyp{}invariant subspaces of the fibre of \(P_\per V\) above \(g\).
\end{theorem}

\begin{proof}
 An element of \(P_\per V\) in the fibre above \(g\) is represented by a pair \((\alpha, \beta)\) with \(\alpha \in P_\per G\) above \(g\) and \(\beta \in L V\).
 Any alternative representative is of the form
  \((\alpha \gamma, \gamma^{-1} \beta)\)
 for some \(\gamma \in L G\).

 Thus the map \(\phi \colon \R \to V\)
 defined by
  \(\phi \coloneqq \alpha \beta\)
 depends only on the element of \(P_\per V\) and not on the choice of representative.
 This satisfies:
 \[
  (\tau \phi)(t) = \phi(t + 1) = \alpha(t + 1) \beta(t + 1) = g \alpha(t) \beta(t) = g \phi(t).
 \]
 Hence \(\tau \phi = g\phi\).

 Conversely, suppose that \(\tau \phi = g \phi\).
 Choose some \(\alpha \in P_\per G\) above \(g\) and define
  \(\beta \coloneqq \alpha^{-1} \phi\).
 Then
  \(\beta(t + 1) = \alpha^{-1}(t) g^{-1} g \phi(t) = \beta(t)\)
 so \(\beta \in L V\).
 Changing \(\alpha\) to \(\alpha \gamma\) changes \(\beta\) to \(\gamma^{-1} \beta\).
 Hence the element in \(P_\per V\) represented by \((\alpha, \beta)\) depends only on \(\phi\).

 Now we consider the polynomial path space.
 We need to show that the fibre of \(P_\pol V\) above \(g\) is the union of the finite dimensional subspaces of the fibre of \(P_\per V\) that are \(D\)\enhyp{}invariant.

 Let \(\xi \in \mf{g}\) be such that \(\exp(\xi) = g\).
 This defines two actions on \(P V\).
 The first is multiplication by \(\eta_{-\xi}\),
  \(\alpha \mapsto \eta_{-\xi} \alpha\),
 which maps \(P_{\per,g} V\) onto \(L V\).
 The second is multiplication by \(\xi\),
  \(\alpha \mapsto \xi \alpha\),
 extending the action of \(\mf{g}\) on \(V\) to \(P V\).
 As \(\xi\) is a finite dimensional operator, it has a minimum polynomial.
 This is true also of the second action on \(P V\).
 Therefore any finite dimensional subspace of \(P V\) is contained in a finite dimensional \(\xi\)\enhyp{}invariant subspace.
 Moreover, the action of \(\xi\) on \(P V\) commutes with that of \(D\) so any finite dimensional \(D\)\enhyp{}invariant subspace of \(P V\) is contained in a finite dimensional subspace that is both \(D\)\enhyp{}invariant and \(\xi\)\enhyp{}invariant.

 Hence as \(\xi\) preserves both \(P_{\per, g} V\) and \(L V\), when considering the union of finite dimensional \(D\)\enhyp{}invariant subspaces in either, it is sufficient to consider those that are in addition \(\xi\)\enhyp{}invariant.

 We shall now show that \(W \subseteq L V\) is \(\xi\) and \(D\)\enhyp{}invariant if and only if \(\eta_\xi W\) is \(\xi\) and \(D\)\enhyp{}invariant.
 This will establish the result.

 The \(\xi\)\enhyp{}invariance is straightforward since \(\xi\) commutes with \(\eta_\xi\).
 Hence \(W \subseteq L V\) is \(\xi\)\enhyp{}invariant if and only if
  \(\eta_\xi W \subseteq P_{\per,g} V\)
 is \(\xi\)\enhyp{}invariant.

 If \(W\) is \(\xi\) and \(D\)\enhyp{}invariant, then consider
  \(\alpha \in \eta_{\pm \xi} W\)
 (the \(\pm\) allows us to consider both directions at once).
 This is of the form \(\eta_{\pm \xi} \beta\)
 for some \(\beta \in W\).
 Then:
 \[
  D \alpha = (D \eta_{\pm \xi}) \beta + \eta_{\pm \xi} (D \beta) = \eta_{\pm \xi} ( \pm \xi \beta + D \beta) \in \eta_{\pm \xi} W.
 \]
 Hence \(\eta_{\pm \xi} W\) is \(D\)\enhyp{}invariant.
\end{proof}

An immediate corollary of this is that the fibres of \(P_\per V\) and of \(P_\pol V\) are \(D\)\enhyp{}invariant.
For \(P_\per V\) this follows from the fact that \(\exp(D) = \tau\) so \(D\) and \(\tau\) commute.
If we wish to emphasise the fibre, we shall refer to \(D\) as \(D_g\).

In the complex case, as \(D_g\) is skew\hyp{}adjoint, any element of the fibre of \(P_\pol V\) above \(g\) is thus the sum of eigenvectors of \(D_g\).

When viewing a fibre of \(P_\per V\) or \(P_\pol V\) as a subspace of \(P V\), the corresponding element \(g \in G\) is not uniquely determined by any one path (contrast with the case of \(P_\per G\) or \(P_\pol G\)).
Thus to keep track of the fibre, we shall often use the notation \((g, \phi)\).

There is an action of \(G\) on \(P_\per V\) and on \(P_\pol V\) given by the following equivalent definitions:
\begin{align*}
 g \cdot \lb \alpha, \beta \rb &%
 = \lb g \alpha, \beta \rb, \\
 g \cdot \lb \alpha, \beta \rb &%
 = \lb g \alpha g^{-1}, g \beta \rb, \\
 g \cdot (h, \phi) &%
 = (g h g^{-1}, g \phi).
\end{align*}
We put in both of the top two descriptions to show that the two actions of \(G\) on \(P_\per G\) (and thus on \(P_\pol G\)) define the same action on \(P_\per V\) (and \(P_\pol V\)).
This action preserves the subbundle \(P_\pol V\) and sends the operator \(D_h\) to \(D_{g h g^{-1}}\).

\subsection{Other Loop Bundles}
\label{sec:otherbundles}

The groups \(L_\pol G\) and \(L G\) act on other linear loop spaces besides the actions on \(L_\pol V\) and \(L V\) considered in the previous section.
Whenever we have such an action then we can define an associated vector bundle over \(G\).
Moreover, whenever we have a map between loop spaces which is \(L_\pol G\)\enhyp{}equivariant or \(L G\)\enhyp{}equivariant then we get a similar map between the corresponding bundles.

In particular, we define the periodic \(L^2\)\enhyp{}path space as
\[
 P_\per^2 V \coloneqq P_\per G \times_{L G} L^2 V.
\]
Using the definitions of section~\ref{sec:linear} we define
\[
 P^2_{\per,r} V \coloneqq P_\pol G \times_{L_\pol G} L^2_{r} V.
\]

\begin{proposition}
The bundle \(P_{\per} V\) has a nuclear locally equivalent co\hyp{}orthogonal structure with structure group \(\gl_\res\).
\end{proposition}

\begin{proof}
The inclusion \(P^2_{\per,r} V \to P_\per V\) has fibrewise dense image and so defines a co\hyp{}orthogonal structure on \(P_\per V\).
As both have structure group \(L_\pol G\) we can simultaneously trivialise both with respect to this group, thus the co\hyp{}orthogonal structure is locally equivalent.
The action on \(L_\pol G\) on \(L^2_r V\) factors through \(\gl_\res\) so the co\hyp{}orthogonal structure has structure group \(\gl_\res\).
Finally, the inclusion \(L^2_r V \to L V\) is a nuclear map so the co\hyp{}orthogonal structure is nuclear.
\end{proof}

In section~\ref{sec:linear} we constructed an isomorphism \(L^2_r V \to L^2 V\) but this isomorphism was not \(L_\pol G\)\enhyp{}equivariant and so we do not automatically get an associated isomorphism of bundles.
It is, however, possible to transfer this isomorphism to the bundles.

\begin{lemma}
 The operator \(\cos_r D\) defines a bundle isomorphism
  \(P_{\pol} V \to P_{\pol} V\)
 which extends to an isomorphism
  \(P^2_{\per,r} V \to P^2_{\per} V\).
\end{lemma}

\begin{proof}
 This is the global version of theorem~\ref{th:ldeiso}.
 All we need to do to apply that theorem is to show that we can choose the operator \(D\) consistently on the fibres of \(P_\pol V\).
 The operator \(D\) of the previous section will do.
\end{proof}

\begin{corollary}
If we enlarge the structure group of \(P_\per V\) to \(E L \gl(V)\) then this co\hyp{}orthogonal structure becomes locally trivial.
\end{corollary}

\begin{proof}
By viewing the original structure group of \(P_\per V\) as \(E L \gl(V)\) we can find a subgroup isomorphic to \(L_\pol U(V)\) which acts unitarily on \(L^2_r V\).
Thus the co\hyp{}orthogonal structure constructed above is locally trivial.
\end{proof}

Once again, there is an action of \(G\) on all of this structure which maps down to the conjugation action on the base.
The action on the vector bundles is as for the vector bundles of smooth and polynomial paths.
This action preserves the extra orthogonal structure as the action of \(G\) on each of the linear spaces involved preserves this structure.

%\newpage
\section{Free Loop Bundles}
\label{sec:free}

The goal of this section is to construct the Hilbert subspace of the vector bundle \(L E \to L M\), where \(E \to M\) is a real or complex vector bundle.
The first part of this construction involves defining the polynomial loop bundle, \(L_\pol E \to L M\), and proving that it is a locally trivial vector bundle modelled on \(L_\pol \F^n\), for \F one of \R or \C.
Once this has been defined, we can thicken it to a bundle \(L^2_r E\) modelled on \(L^2_r \F^n\).
This defines the locally equivalent co\hyp{}orthogonal structure of \(L E\).
We then construct a bundle isomorphism \(L^2_r E \to L^2 E\) similar to the linear one of section~\ref{sec:linco}.
To prove theorem~\ref{th:coriem} we apply this construction with \(E = T M\) using the canonical identification of \(L T M\) with \(T L M\).

In section~\ref{sec:polprop} we discuss the basic properties of the polynomial loop bundle, and thus of \(L^2_r E\).
In particular we consider the action of the group of diffeomorphisms of the circle.
The natural action on \(L E\) does not preserve the polynomial subbundle but it can be modified to an action which does.

\subsection{Notation}
\label{sec:loop}

Let \(M\) be a smooth finite dimensional manifold without boundary.
Let \(G\) be one of \(U_n\), \(S U_n\), or \(S O_n\).
Let \F be the corresponding field.
Let \(Q \to M\) be a principal \(G\)\enhyp{}bundle.
Let
 \(E \coloneqq Q \times_G \F^n\)
be the corresponding vector bundle.
As \(G\) preserves the inner product on \(\F^n\), \(E\) carries a fibrewise inner product.
Let \(\nabla\) be a covariant differential operator on \(E\) coming from a connection on \(Q\).

We think of a point in a fibre \(Q_p\) as being an isometry \(\F^n \to E_p\).
We shall also use the \emph{adjoint bundle} associated to \(Q\),
 \(Q^\ad \coloneqq Q \times_{\text{conj}} G\)
where \(G\) acts on itself by conjugation.
This is a bundle of groups.
A point in a fibre \(Q^\ad_p\) is an isometry of \(E_p\) to itself.

It is a standard result that the loop and path spaces of \(E\), \(L E\) and \(P E\), form vector bundles over, respectively, the loop and path spaces of \(M\), \(L M\) and \(P M\), with frame bundles the loop and path spaces of \(Q\), \(L Q\) and \(P Q\), and adjoint bundles the loop and path spaces of \(Q^\ad\), \(L Q^\ad\) and \(P Q^\ad\).

It is not strictly relevant, but should be remarked that whilst \(L M\) (and the other loop spaces) are infinite dimensional smooth manifolds, \(P M\) (and the other path spaces) are not.
They are smooth \emph{spaces} but are not locally linear.
The reason that this does not concern us is that the path spaces are used merely as a ``background canvas'' on which the interesting structure is painted.
See \cite{akpm} for more details on infinite dimensional smooth spaces.

\subsection{The Holonomy Operator}
\label{sec:holonmy}

Recall that we can view \(M\) in \(P M\) as the subspace of constant paths and \(L M\) in \(P M\) as the subspace of periodic paths (of period \(1\)).
For \(X\) each of \(E\), \(Q\), and \(Q^\ad\), define a bundle \(P^M X \to M\) by restricting \(P X\) to this subspace.
The fibre of \(P^M X\) above \(p \in M\) is thus \(P (X_p)\).
In terms of the original bundles \(E\), \(Q\), and \(Q^\ad\) we have the following descriptions.
\begin{align*}
 P^M E &%
 = E \otimes P \F, \\
 &%
 = Q \times_G P \F^n, \\
 P^M Q &%
 = Q \times_{G} P G, \\
 P^M Q^\ad &%
 = Q \times_{\text{conj}} P G.
\end{align*}
As with \(M\) inside \(L M\) and \(P M\), \(G\) sits inside \(L G\) and \(P G\) as the constant loops.
In the middle line, the action is as a subgroup, in the third line the action is via conjugation.

The covariant differential operator defines a parallel transport operator.
To describe this then for \(t \in \R\) let \(e_t \colon P M \to M\)
be the map which evaluates at time \(t\) and let \(X^t \to P M\) be the bundle \(e_t^* X\); thus
 \(X^t_\gamma = X_{\gamma(t)}\).
We use the same notation for the restriction of these bundles to \(L M\).
The covariant differential operator defines three compatible families of bundle maps
 \(\psi_X^t \colon X^t \to P X\).
The properties of these maps are:
\begin{align}
 \psi_E^t(p q w) &%
 = \psi_{Q^\ad}^t(p) \psi_Q^t(q) w &%
 &
 p \in Q^{\ad,t}_\gamma,\, q \in Q^t_\gamma,\, w \in \F^n \subseteq P \F^n.
 \label{eq:psitriple}
 \\
 \psi_X^t e_t \psi_X^s &%
 = \psi_X^s.
 \label{eq:psibreak}
 \\
 e_{s + 1} \psi_X^{t + 1} &%
 = e_s \psi_X^t &%
 &
 \text{over } L M.
 \label{eq:psiloop}
\end{align}
For the second, note that \(e_t \psi_X^s\) is a map from \(X^s\) to \(X^t\).
This compatibility relation is the statement that if one parallel transports from time \(s\) to time \(t\) and then on from time \(t\) to some when else, it is the same as transporting straight from \(s\) to ones final time.
For the last, over \(L M\) then \(X^{t + 1} = X^t\) so the domains and codomains of these maps are the same.
This property is then an application of the fact that the parallel transport operator is intrinsic to \(M\), therefore the parallel transport from \(X^t\) to \(X^s\) is the same as that from \(X^{t + 1}\) to \(X^{s + 1}\).

Let
 \(P^{M,t} X \coloneqq e_t^* P^M X\),
so
 \(P^{M,t}_\alpha X = P(X_{\alpha(t)})\).
The parallel transport operators extend to bundle equivalences:
\begin{equation}
 \label{eq:buneq}
 \Psi_X^t \colon P^{M,t} X \to P X, 
\end{equation}
with the property that
 \(e_s( \Psi_X^t \alpha) = (e_s \psi_X^t) (\alpha(s))\).
Note that these equivalences have been chosen such that \((\Psi_X^t \alpha)(s)\)
always lies in \(X^s\) no matter which \(t\) was the starting point.

Let \(P^L X\) be the restriction of \(P X\) to \(L M\).
Thus \(P^L X\) consists of paths in \(X\) which project down to loops in \(M\).
There is an obvious inclusion of \(L X\) in \(P^L X\) and it is straightforward to recognise this submanifold: \(L X\) consists of those paths in \(P^L X\) which are themselves periodic.
Note that for any path \(\beta\) in \(P^L X\) then \(\beta(t + 1)\) and \(\beta(t)\) both lie in the same fibre of \(X \to M\).

Thus in the right hand side of~\eqref{eq:buneq} (restricted to \(L M\)), it is straightforward to recognise the subbundles consisting of the loops.
We wish to transfer this recognition principle to the left hand side of~\eqref{eq:buneq}.
We do this using the \emph{holonomy operator}.

\begin{defn}
 On \(L M\), define the fibrewise operators
  \(h_X \colon X^0 \to X^0\)
 by \(h_X = e_1 \psi_X^0\).
\end{defn}

Over \(P M\), \(e_1 \psi_X^0\) is a map \(X^0 \to X^1\).
Over \(L M\) then \(X^0 = X^1\) so \(h_X\) is as defined.
The fibres of \(Q^\ad\) act on each of \(E\), \(Q\), and \(Q^\ad\): on \(E\) the action is by definition, on \(Q\) and on \(Q^\ad\) by composition.

\begin{lemma}
 The operator \(h_E\) is a section of \(Q^{\ad,0}\).
 The operators \(h_E\), \(h_Q\), and \(h_{Q^\ad}\) satisfy:
  \(h_{Q^\ad}(p) h_E = h_E p\),
 and \(h_Q(q) = h_E q\).
 Thus \(h_E\) determines both \(h_Q\) and \(h_{Q^\ad}\).
\end{lemma}

\begin{proof}
 Since \(e_1 \psi^0_E\) is a fibrewise isometry \(E^0 \to E^0\), it is a section of \(Q^{\ad,0}\).
 Then from~\eqref{eq:psitriple}, for \(p \in Q^{\ad,0}\), \(q \in Q^0\), \(v \in E^0\), and
  \(w \in \F^n \subseteq P \F^n\):
 \begin{align*}
  (h_E p)v &%
  = (e_1 \psi^0_E p ) v \\
  &%
  = e_1 (\psi^0_E (p v)) \\
  &%
  = e_1 (\psi^0_{Q^\ad}(p) \psi^0_E(v)) &%
  &
  \text{by equation~\eqref{eq:psitriple}} \\
  &%
  = (e_1 \psi^0_{Q^\ad})(p) (e_1 \psi^0_E)(v) \\
  &%
  = h_{Q^\ad}(p) h_E(v).
  \\
  (h_E q)w &%
  = (e_1 \psi^0_E q) w \\
  &%
  = e_1( \psi^0_E( q w)) \\
  &%
  = e_1( \psi^0_Q(q) w ) &%
  &
  \text{by equation~\eqref{eq:psitriple}}\\
  &%
  = (e_1 \psi^0_Q)(q) w \\
  &%
  = h_Q(q) w.
  &&\qedhere
 \end{align*}
\end{proof}

\begin{lemma}
  \(e_{t + 1} \psi_X^0 = e_t \psi^0_X h_X\).
\end{lemma}

\begin{proof}%
 \begin{align*}
  e_{t + 1} \psi_X^0 &%
  = e_{t + 1} \psi^1_X e_1 \psi^0_X &%
  &
  \text{by equation~\eqref{eq:psibreak}} \\
  &%
  = e_t \psi^0_X e_1 \psi^0_X &%
  &
  \text{by equation~\eqref{eq:psiloop}} \\
  &%
  = e_t \psi^0_X h_X. &&\qedhere
 \end{align*}
\end{proof}

\begin{corollary}
 Under the bundle isomorphism of equation~\eqref{eq:buneq}, the subbundle \(L X\) of \(P^L X\) corresponds to:
 \[
  \{\alpha(t) \in P^{M,0} X \colon h_X \alpha(t + 1) = \alpha(t)\}.
 \]
\end{corollary}

\begin{proof}
 An element \(\alpha \in P^{M,0} X\)
 is mapped to a loop in \(P^L X\) if and only if
  \((\Psi^0_X \alpha)(t + 1) = (\Psi^0_X \alpha)(t)\)
 for all \(t \in \R\).
 The left hand side of this simplifies to:
 \[
  e_{t + 1} (\Psi^0_X \alpha) = (e_{t + 1} \psi^0_X)(\alpha(t + 1)) = (e_t \psi^0_X)(h_X \alpha(t + 1))
 \]
 whilst the right hand side simplifies to:
 \[
  e_t (\Psi^0_X \alpha) = (e_t \psi^0_X)(\alpha(t)).
 \]

 Since
  \(e_t \psi^0_X \colon X^0 \to X^t\)
 is an isomorphism, this implies that \(\Psi^0_X \alpha\) is a loop if and only if
  \(h_X \alpha(t + 1) = \alpha(t)\)
 for all \(t \in \R\).
\end{proof}

For \(X\) each of \(E\), \(Q\), and \(Q^\ad\) let \(Y\) be \(\F^n\) or \(G\) as appropriate and define
\[
 P^M_\per X \coloneqq Q \times_G P_\per Y.
\]
As \(P_\per Y \to G\) is a \(G\)\enhyp{}equivariant bundle where the \(G\) action on the base is by conjugation, \(P^M_\per X\) is a fibre bundle over \(Q^\ad\).
Via the evaluation map \(e_0 \colon L M \to M\)
we therefore get bundles
 \(e_0^* P^M_\per X \to Q^{\ad,0}\).

For a section
 \(\chi \colon L M \to Q^{\ad,0}\)
and \(X\) each of \(E\), \(Q\), and \(Q^\ad\) define
\[
 L^\chi X \coloneqq \chi^* e_0^* P^M_\per X.
\]

\begin{corollary}
 For \(X\) each of \(Q\), \(Q^\ad\), and \(E\), \(\Psi^0_X\) restricts to a bundle isomorphism
  \(L^{h_E^{-1}} X \to L X\). \noproof
\end{corollary}

\subsection{Other Loop Bundles}
\label{sec:polybundles}

The advantage of the above descriptions of \(L E\), \(L Q\), and \(L Q^\ad\) in terms of the holonomy operator is that they can be generalised using other bundles over \(G\).
All that is needed is to have a \(G\)\enhyp{}equivariant bundle over \(G\) with respect to the conjugation action on the base.
In particular, we can take the bundles that we constructed in section~\ref{sec:linear}.

\begin{defn}
 Define the following spaces
 \begin{align*}
  P^M_\pol E &%
  \coloneqq Q \times_G P_\pol \F^n, \\
  P^M_\pol Q &%
  \coloneqq Q \times_G P_\pol G, \\
  P^M_\pol Q^\ad &%
  \coloneqq Q \times_{\text{conj}} P_\pol G, \\
  P^{2,M}_\per E &%
  \coloneqq Q \times_G P^2_{\per} \F^n, \\
  P^{2,M}_{\per,r} E &%
  \coloneqq Q \times_G P^2_{\per,r} \F^n,
 \end{align*}
 where in the last two lines we carry over the associated orthogonal structure.
\end{defn}

For a section
 \(\chi \colon L M \to Q^{\ad,0}\)
we therefore have associated bundles
\[
 L^\chi_\type X \coloneqq \chi^* e_0^* P^M_\type X.
\]
The following is immediate.

\begin{lemma}
 The bundles \(L^\chi_\type X\) are locally trivial fibre bundles modelled on \(L_\type Y\). \noproof
\end{lemma}

When the original linear loop space was a subspace of smooth loops and we take the inverse of the holonomy operator as our section then we can use \(\Psi_X^0\) to transfer this bundle to a subbundle of the bundle of smooth loops.

\begin{defn}
 The \emph{polynomial loop bundles}, \(L_\pol X\), for \(X\) each of \(Q\), \(Q^\ad\), and \(E\) are defined to be the images in \(L X\) of \(L^{h_E^{-1}}_\pol X\) under the map \(\Psi^0_X\).

 Similarly, we define \(L^2_{r} E\) to be the image in \(L E\) of \(L^{2,h_E^{-1}}_{r} E\)
 under \(\Psi^0_E\).
\end{defn}

We therefore have the following result.

\begin{proposition}
 The polynomial loop bundles are locally trivial with \(L_\pol Q\) a \(L_\pol G\)\enhyp{}principal bundle, \(L_\pol Q^\ad\) a bundle of groups modelled on \(L_\pol G\), and \(L_\pol E\) a vector bundle modelled on \(L_\pol \F^n\).
 Moreover:
 \begin{align*}
  L_\pol Q^\ad &%
  = L_\pol Q \times_{\text{conj}} L_\pol G, \\
  L_\pol E &%
  = L_\pol Q \times_{L_\pol G} L_\pol \F^n, \\
  L Q &%
  = L_\pol Q \times_{L_\pol G} L G, \\
  L Q^\ad &%
  = L_\pol Q \times_{\text{conj}} L G, \\
  L E &%
  = L_\pol Q \times_{L_\pol G} L \F^n, \\
  L^2_{r} E &%
  = L_\pol Q \times_{L_\pol G} L^2_{r} \F^n, \\
  L^2 E &%
  = L_\pol Q \times_{L_\pol G} L^2 \F^n.
 \end{align*}
\end{proposition}

\begin{corollary}
 The bundle \(L E \to L M\) has a nuclear locally equivalent co\hyp{}orthogonal structure with structure group \(\gl_\res\).
 This structure depends naturally on the inner product and compatible connection on \(E\). \noproof
\end{corollary}

The bundle \(L_\pol E\) has a more concrete description in terms of the connection on \(E\).
For any path
 \(\gamma \colon \R \to M\),
the connection on \(E\) defines a covariant differential operator
 \(D_\gamma \colon \Gamma_\R(\gamma^* E) \to \Gamma_\R(\gamma^* E)\);
that is,
 \(D_\gamma \colon P_\gamma E \to P_\gamma E\).
As the map \(\Psi_E^0\) was constructed using parallel transport, it (rather, its inverse) takes \(D_\gamma\) to the operator \(\diff{}{t}\) acting on \(P^{M,0} E\).
If \(\gamma\) happens to be a loop, \(D_\gamma\) restricts to an operator on \(L_\gamma E\).
As \(\Psi_E^0\) identifies \(L_\gamma E\) with the fibre of
 \(P_\per E \to Q^{\ad,0}\)
above \(h_E^{-1}(\gamma)\), it takes \(D_\gamma\) to the operator \(D_{h_E^{-1}(\gamma)}\).

Hence \(L_{\pol, \gamma} E\) can be constructed from the action of \(D_\gamma\) on \(L_\gamma E\) in the same fashion as \(P_{\pol,g} \F^n\) from \(P_{\per,g} \F^n\), namely as the union of the finite dimensional \(D_\gamma\)\enhyp{}invariant subspaces of \(L_\gamma E\).
In the complex case, \(L_{\pol, \gamma} E\) is the span of the eigenvalues in \(L_\gamma E\) of \(D_\gamma\).

This description is more in the spirit of \cite{jm2}.
However, one still has to show that the resulting object is a locally trivial vector bundle over \(L M\) and the simplest method for that is by considering principal bundles as above.

We can give an explicit formula for the inner product on the fibres of \(L^2_r E\) in terms of this operator \(D_\gamma\).
Let \(\ipc\) be the inner product on \(L^2 E\).
The operator \(D_\gamma\) defines an isomorphism
 \(\cos_r D_\gamma \colon L^2_r E \to L^2 E\)
and so the inner product on \(L^2_r E\) is given by
\[
 \ipv{\alpha}{\beta}_\gamma = \ip{(\cos_r D_\gamma)^{-1} \alpha}{(\cos_r D_\gamma)^{-1} \beta}_\gamma.
\]

\subsection{Properties of the Polynomial Bundle}
\label{sec:polprop}

The construction of the polynomial loop bundle started from a connection on the original bundle over \(M\).
However, it only actually used the map
 \(\psi_X^0 \colon X^0 \to P X\)
defined by the parallel transport operator.
Thus as far as the polynomial loop bundle is concerned, having a connection is overkill.
The connection is useful, though, as it implies that the polynomial loop bundle came from structure on the original manifold \(M\) and thus one can hope for more structure on the polynomial loop bundle than has yet been described.
In this section, we shall investigate this.

Before examining the interesting properties of the polynomial loop bundle, we list some basic ones that are fairly obvious:
\begin{proposition}
 Let \(M\) be a finite dimensional smooth manifold, \(E_1, E_2 \to M\) finite dimensional vector bundles over the same field with inner products and connections compatible with the inner products.
 \begin{enumerate}
 \item
  Let \(E = E_1 \oplus E_2\) orthogonally and equip \(E\) with the direct sum connection.
  Then
   \(L_\pol E = L_\pol E_1 \oplus L_\pol E_2\).

 \item
  Suppose that \(E_1\) is real, then
   \(L_\pol (E_1 \otimes \C) = (L_\pol E_1) \otimes \C\).

 \item
  Suppose that \(E_1\) is complex, then
   \(L_\pol ({E_1}_\R) = (L_\pol E_1)_\R\).

 \item
  Let
   \(\psi \colon E_1 \to E_2\)
  be a bundle isomorphism which preserves the inner products and connections.
  Then \(\psi\) defines an isomorphism
   \(L_\pol \psi \colon L_\pol E_1 \to L_\pol E_2\).

 \item
  Suppose that \(E_1\) with its inner product is a sub-bundle of \(E_2\) and that the covariant differential operator on \(E_1\) is of the form \(p \nabla\) where \(p \colon E_2 \to E_1\)
  is the orthogonal projection and \(\nabla\) is the covariant differential operator on \(E_2\).
  Then it is not necessarily the case that
   \(L_\pol E_1 = L_\pol E_2 \cap L E_1\).
 \end{enumerate}
\end{proposition}

\begin{proof}
 Only the last of these is not immediate from the construction.
 Let \(E_2\) be the bundle \(S^1 \times \C^2\) and \(E_1\) the bundle \(S^1 \times \C^1\).
 Include \(E_1\) in \(E_2\) via the map
  \((t, 1) \to (t, \frac1{\sqrt{2}} ( 1, e^{2 \pi i t}) )\).

 The loop space of \(E_1\) is \(L S^1 \times L \C\) and of \(E_2\) is \(L S^1 \times L \C^2\).
 The polynomial loop space of \(E_2\) is
  \(L S^1 \times L_\pol \C^2\).
 The inclusion \(L E_1 \to L E_2\) is given by:
 \[
  (\gamma, \beta) \to (\gamma, \frac1{\sqrt{2}} (\beta, e^{2 \pi i \gamma(t)} \beta)).
 \]
 Therefore
  \(L E_1 \cap L_\pol E_2\)
 consists of those loops \(\beta\) such that both \(\beta\) and
  \(e^{2 \pi i \gamma} \beta\)
 are polynomials.
 We can choose \(\gamma\) such that whenever \(\beta\) is polynomial then
  \(e^{2 \pi i \gamma} \beta\)
 is not.
 Hence there is some \(\gamma\) such that above \(\gamma\) the fibres of \(L E_1\) and \(L_\pol E_2\) intersect trivially.
\end{proof}

The advantage of having the polynomial structure defined using a connection on the original bundle is the relationship with the diffeomorphism group of the circle.
For
 \(\sigma \colon S^1 \to S^1\)
smooth (not necessarily a diffeomorphism),
 \(\gamma \colon S^1 \to M\),
and
 \(\alpha \in L_\gamma E\),
the following is a simple application of the chain rule:
\begin{equation}
 \label{eq:condiff}
 D_{\gamma \circ \sigma} (\alpha \circ \sigma) = \left((D_\gamma \alpha) \circ \sigma \right) \sigma',
\end{equation}
where
 \(\sigma' \colon S^1 \to \R\)
is such that
 \(d \sigma (\diff{}{t}) = \sigma' \diff{}{t}\).

From this formula, two results can be derived:
\begin{proposition}
 \begin{enumerate}
 \item
  The action of \(\Diff(S^1)\) on \(L E\) does not preserve the sub-bundle \(L_\pol E\).
  The subgroup of \(\Diff(S^1)\) which does preserve the sub-bundle \(L_\pol E\) is \(S^1 \rtimes \Z/2\) where the non-trivial element in the \(\Z/2\)\enhyp{}factor is the diffeomorphism \(t \to -t\).

 \item
  Let \(\nabla^a\) and \(\nabla^b\) be two different connections on \(E\).
  The two polynomial bundles so defined are different.
 \end{enumerate}
\end{proposition}

\begin{proof}
 We shall consider the complex case so that we may talk about eigenvectors and eigenvalues of \(D_\gamma\).
 The real case may be deduced from this.

 \begin{enumerate}
 \item
  For this, consider the situation over a constant loop.
  There, \(L E\), resp.\ \(L_\pol E\), is \(E \otimes L \C\), resp.\
   \(E \otimes L_\pol \C\).
  The action of \(\Diff(S^1)\) on \(L E\) is given by its action on \(L \C\).
  Thus if
   \(\sigma \in \Diff(S^1)\)
  preserves \(L_\pol E\) then it must preserve \(L_\pol \C\) within \(L \C\).

  The map \(t \to e^{2 \pi i t}\) lies in \(L_\pol \C\).
  It is also the identification of \(S^1\) with \T. Under \(\sigma\) this transforms to
   \(t \to e^{2 \pi i \sigma(t)}\).
  As \(\sigma\) is a diffeomorphism of \(S^1\), this map must still be an identification of \(S^1\) with \T. The only polynomials which do this are those of the form
   \(t \to \nu e^{\pm 2 \pi i t}\)
  for \(\nu \in \T\).
  Hence if
   \(\sigma \in \Diff(S^1)\)
  preserves \(L_\pol E\) within \(L E\) then
   \(\sigma \in S^1 \rtimes \Z/2\).

  The converse is direct from the equation~\ref{eq:condiff} since if
   \(\sigma \in S^1 \rtimes \Z/2\)
  then \(\sigma' = \pm 1\) so:
  \[
   D_{\gamma \circ \sigma} (\alpha \circ \sigma) = \pm (D_\gamma \alpha) \circ \sigma.
  \]
  Hence \(\sigma\) maps eigenvectors of \(D_\gamma\) to eigenvectors of
   \(D_{\gamma \circ \sigma}\)
  and thus preserves \(L_\pol E\).

 \item
  As \(\nabla^a\) and \(\nabla^b\) are different, there will be some loop \(\gamma\) such that \(D_\gamma^a\) and \(D_\gamma^b\) differ.
  The difference will be a section \(\Phi\) of the bundle
   \(\mf{u}(\gamma^* E) \to S^1\),
  in other words an element of \(L_\gamma \mf{u}(E)\).

  If
   \(L_{\pol,\gamma}^a E = L_{\pol, \gamma}^b E\)
  then both are preserved under \(D_\gamma^a\) and \(D_\gamma^b\), hence under their difference.
  Thus \(\Phi\) must be an element of
   \(L_\pol \mf{u}(E)\).

  By examining equation~\ref{eq:condiff}, we see that under the action of a smooth self-map \(\sigma\) of the circle, \(\Phi\) transforms to
   \((\Phi \circ \sigma) \sigma'\).
  It is then a simple matter to find \(\sigma\) such that this is no longer a polynomial.
  Hence even if we were unlucky enough initially to choose a loop \(\gamma\) with
   \(L_{\pol,\gamma}^a E = L_{\pol,\gamma}^b E\)
  then we can find some other loop
   \(\gamma \circ \sigma\)
  over which the fibres of the polynomial bundles differ.
  \qedhere
 \end{enumerate}
\end{proof}

It is straightforward to show that the result about the action of \(\Diff(S^1)\) on \(L_\pol E\) generalises to the statement that the subgroup of \(\Diff(S^1)\) which preserves
 \(L^?
E\) is
 \(\Diff(S^1) \cap L^?
\C\) where the ``\(?\)'' represents some class of regularity of loop.

In the light of this result, it is perhaps surprising that there is an action of \(\Diff(S^1)\) on \(L_\pol E\) which covers the standard action of \(\Diff(S^1)\) on \(L M\).
This comes about because the \(\Diff(S^1)\)\enhyp{}action preserves the parallel transport operator.
Since all else was derived from that, we can make \(\Diff(S^1)\) act on \(L_\pol E\).

We start with the group \(\Diff_0^+(S^1)\) of orientation and basepoint preserving diffeomorphisms.
Since the whole diffeomorphism group is the semi-direct product of this with \(S^1 \rtimes \Z/2\), an action of this group together with the above action of \(S^1 \rtimes \Z/2\) will give an action of the whole diffeomorphism group.

An element of \(\Diff_0^+(S^1)\) lifts canonically to an element of \(\Diff_0^+(\R)\).
The image consists of those diffeomorphisms of \R which satisfy
 \(\sigma(t + 1) = \sigma(t) + 1\).
This allows \(\Diff_0^+(S^1)\) to act on paths as well as loops.

Let
 \(\sigma \in \Diff_0^+(S^1)\).
Recall that the bundle \(P^{M,0} E \to L M\) has fibre
 \(P^{M,0}_\gamma E = P( E_{\gamma(0)})\).
Thus as
 \(\gamma \circ \sigma(0) = \gamma(0)\),
the bundles \(P^{M,0} E\) and \(\sigma^*(P^{M,0} E)\) are genuinely the same bundle.
The bundle \(P^L E\), meanwhile, has fibre
 \(P^L_\gamma E = \Gamma(\gamma^* E)\).
Thus there is a natural isomorphism
 \(P^L E \to \sigma^*(P^L E)\)
given by
 \(\alpha \to \alpha \circ \sigma\).

With these two isomorphisms, the square:
\[
 \xymatrix{ P^{M,0} E \ar[r]^{\Psi_E} \ar@{=}[d]&
  P^L E \ar[d]^{\sigma} \\
  P^{M,0} E \ar[r]^{\Psi_E} &
  P^L E }
\]
does not commute.
To make it commute, we need to transfer one action of \(\sigma\) from one side to the other.
Clearly, the action of \(\sigma\) on \(P^L E\) restricts to the standard action on \(L E\) which we know does not preserve \(L_\pol E\).

It is also true that the action of \(\sigma\) on \(P^{M,0} E\) preserves \(L^{M,0,h_E^{-1}} E\) and \(L^{M,0,h_E^{-1}}\).
Thus is because the holonomy operator \(h_E\) is equivariant under the action of \(\Diff_0^+(S^1)\).
Therefore, the action of \(\sigma\) on \(P^{M,0} E\) when transferred to \(P^L E\) also restricts to an action on \(L E\) and on \(L_\pol E\).

In formul\ae, the two actions of \(\Diff^+_0(S^1)\) are as follows: any element of \(P_\gamma E\) can be written as
 \(\sum_j f^j \psi^0_E v_j\)
where \(\{v_1, \dotsc, v_n\}\)
is a basis for \(E_{\gamma(0)}\).
The usual action is:
\[
 \sigma \left( \sum_j f^j \psi^0_E v_j \right) = \sum_j (f^j \circ \sigma) \psi^0_E v_j
\]
and the new action is:
\[
 \sigma \left( \sum_j f^j \psi^0_E v_j \right) = \sum_j f^j \psi^0_E v_j.
\]

One way to make the distinction between the two actions is to have two views of the bundle \(L E \to L M\).
In one, a fibre \(L_\gamma E\) is inextricably linked to the points of \(\gamma(S^1)\).
In the other, the fibre \(L_\gamma E\) is linked only to the map \(\gamma\).
In the former, reparametrising the loop \(\gamma\) does not change \(\gamma(S^1)\) and so the fibres \(L_\gamma E\) and
 \(L_{\gamma \circ \sigma} E\)
are closely related.
Any reasonable\emhyp{}in this view\emhyp{}group action must preserve this relationship.
In the latter view, reparametrising the loop \(\gamma\) changes it and so there is no intrinsic relationship between the fibres \(L_\gamma E\) and
 \(L_{\gamma \circ \sigma} E\).
Therefore there is no special relationship for a reasonable group action to preserve.

%\newpage
\section{The Loop Space of a String Manifold}
\label{sec:string}

\subsection{Spin Bundles}

In finite dimensions, a spin bundle is an oriented bundle with an orthogonal structure together with a lift of the structure group from \(S O_n\) to a certain \(\Z/2\)\enhyp{}central extension, \(\spin_n\).

This can be generalised to infinite dimensions.
There is a group \(S O_J\) which takes the place of \(S O_n\).
This has a central extension \(\spin_J\) similar to the central extension \(\spin_n\) of \(S O_n\), although by \(S^1\) rather than \(\Z/2\).
The group \(S O_J\) acts naturally on a Hilbert space and is the identity component of another group written \(O_J\).
The groups \(S O_J\) and \(O_J\) are deformation retracts of groups \(S \gl_J\) and \(\gl_J\).
These groups all act naturally on a standard separable real Hilbert space.

The reason for the notation for \(S O_J\) and the other groups is that their definition depends on a choice of skew\hyp{}adjoint operator \(J\) acting on a real Hilbert space, say \(H\).
The definition is very similar to that of the group \(\gl_\res(H)\) of section~\ref{sec:polypol} except that the group \(\gl_\res(H)\) was defined in terms of an operator on the \emph{complexification} of \(H\).
The distinction is important.
We know that the completed tangent bundle of a loop space has a \(S O_\res\)\enhyp{}structure but this is not always the same as an \(S O_J\)\enhyp{}structure.
It can be shown that there is a choice of operator \(J\) on \(L^2(S^1, \R^n)\) such that \(S O_J = S O_\res\) if and only if \(n\) is even.

\begin{defn}
Let \(X\) be a smooth manifold.
\begin{enumerate}
\item Let \(\zeta \to X\) be a bundle of Hilbert spaces over \(X\) with structure group \(S O_J\).
A \emph{spin structure} on \(\zeta\) consists of the following data.

\begin{enumerate}
\item A lift of the structure group of \(\zeta\) from \(S O_J\) to \(\spin_J\).

\item A connection on the principal \(\spin_J\)\enhyp{}bundle.
If a connection on \(\zeta\) has already been specified the connection on the \(\spin_J\)\enhyp{}bundle should be a lift of that on \(\zeta\).
\end{enumerate}

\item Let \(\xi \to X\) be a vector bundle.
A \emph{(co\hyp{})spin structure} on \(\xi\) consists of the following data.

\begin{enumerate}
\item A (co\hyp{})orthogonal structure on \(\xi\) with structure group \(\gl_J\).

\item A spin structure on the associated bundle of Hilbert spaces, where the orthogonal structure is used to reduce the structure group of the Hilbert spaces from \(\gl_J\) to \(S O_J\).
\end{enumerate}

\item We further classify (co\hyp{})spin structures according to the classification of their (co\hyp{})orthogonal structures.

\item A \emph{(co\hyp{})spin manifold} is a manifold together with a choice of (co\hyp{})spin structure on its tangent bundle.
\end{enumerate}
\end{defn}

\begin{remark}
An astute student of finite dimensional spin theory will have noticed a significant departure in this definition.
In finite dimensions, the connection on the spin bundle is assumed to be a lift of a pre\hyp{}existing connection on the original bundle.
Here, we do not assume that such a connection exists (though if it does we require the lift).
Whilst we could easily include this assumption for the bundle of Hilbert spaces, it would not necessarily be possible in the more general case.
To have the connection on the spin bundle related to a connection on the original bundle we would have to be able to simultaneously trivialise the original bundle and \emph{orthogonally} trivialise the Hilbert bundle: namely, to have a \emph{locally trivial} (co\hyp{})orthogonal structure.
Obviously, it is more desirable to have the connection be related\emhyp{}in some fashion\emhyp{}to structure on the original bundle but we do not insist on this being a requirement.
\end{remark}

\subsection{String Manifolds}

There is an obstruction theory to a manifold being spin.
In finite dimensions, one needs the first two Stiefel\enhyp{}Whitney classes to vanish.
In infinite dimensions one needs a certain class in \(H^1(X; \Z/2)\) to vanish (to allow the reduction from \(O_\res\) to \(S O_\res\)) and a class in \(H^3(X;\Z)\) (to allow the lift).
When \(X = L M\), these classes are respectively the transgression of the second Stiefel\enhyp{}Whitney class of \(M\) and a certain class
 \(\lambda \in H^4(M;\Z)\)
which satisfies \(2 \lambda = p_1(M)\).

An interesting question to ask is what structure on \(M\) defines a spin structure on \(L M\).
There are several ways to answer this, \cite{sspt} being an important one which also gives rise to much more structure.
Let us outline a more modest approach.

Let \(M\) be an oriented Riemannian manifold.
Let \(Q \to M\) be the principal \(S O_n\)\enhyp{}bundle of the tangent bundle.
The class
 \(\lambda \in H^4(M;\Z)\)
is the obstruction to choosing a class in \(H^3(Q;\Z)\) which restricts to the canonical generator of \(H^3(S O_n; \Z)\) on each fibre of \(Q \to M\).
Now a class in the third integral cohomology class of a space is represented by a gerbe so let us choose a gerbe on \(Q\) which represents this class in \(H^3(Q;\Z)\).
As explained in \cite{jb}, there is a natural way to use a gerbe to define a line bundle on the corresponding loop space.
Thus we get a line bundle on \(L Q\).
If we have chosen our gerbe well, this line bundle\emhyp{}or rather its unit circle\emhyp{}defines the required lift from \(S O_\res\) to \(\spin_\res\).

The Levi\hyp{}Civita connection on \(M\) loops to a connection on \(L M\), which is also orthogonal and torsion\hyp{}free so it is reasonable to also call this the Levi\hyp{}Civita connection.
Thus we just need to lift this connection to our principal \(\spin_\res\)\enhyp{}bundle.
As is also explained in \cite{jb}, gerbes can be equipped with connections and this structure loops to a connection on the corresponding line bundle.
If we choose this connection data properly on the gerbe, the connection on the corresponding line bundle combines with the Levi\hyp{}Civita connection on \(L M\) to define a connection on the \(\spin_\res\)\enhyp{}bundle.
Hence we have a spin structure on \(L M\).

Whatever structure one chooses, the following proposition is immediate.

\begin{proposition}
 Let \(M\) be a Riemannian manifold of finite even dimension.
 Suppose that we can give \(M\) the structure needed to make \(L M\) a spin manifold with its usual orthogonal structure.
 Then the same structure makes \(L M\) a co\hyp{}spin manifold with the co\hyp{}orthogonal structure constructed in this paper.
\end{proposition}

\begin{proof}
 By assumption, the bundle \(L^2 T M\) admits a spin structure.
 We transfer this to \(L^2_r T M\) via the isometric isomorphism \(L^2_r T M \cong L^2 T M\) which is an isomorphism of \(S O_J\)\enhyp{}bundles.
 Note that as the dimension of \(M\) is even, \(S O_\res = S O_J\).
\end{proof}

Let us say that a manifold \(M\) is a \emph{string} manifold if it equipped with sufficient structure to define a spin structure on \(L M\), and thus also a co\hyp{}spin structure.
Using the work of \cite{math/0809.3104} we obtain the following result.

\begin{corollary}
Let \(M\) be a string manifold.
Then \(L M\) admits a \(S^1\)\enhyp{}equivariant Dirac operator. \noproof
\end{corollary}

\mybibliography{arxiv,articles,books,misc}

 \end{document}